\documentclass[11pt]{amsart} % use larger type; default would be 10pt

\usepackage{caption}
\usepackage{subcaption}
\usepackage{tikz}
\usepackage{graphicx,color,nicefrac} 
\usepackage{amssymb, amsmath, amsbsy,mathtools}
\usepackage{amsmath,amsfonts}
\usepackage[normalem]{ulem}
\usepackage[export]{adjustbox}
\newtheorem{theorem}{Theorem}[section]
\newtheorem{lemma}[theorem]{Lemma}

\newcommand{\ei}{\boldsymbol{e}_i}
\newcommand{\ej}{\boldsymbol{e}_j}
\newcommand{\bx}{\boldsymbol{x}}
\newcommand{\by}{\boldsymbol{y}}
\newcommand{\bz}{\boldsymbol{z}}
\newcommand{\bm}{\boldsymbol{m}}
\newcommand{\sHoneper}{\sH^1_{\operatorname{per}}}
\newcommand{\dd}{\mathrm{d}}

\newcommand{\R}{\mathbb{R}}
\newcommand{\h}{\boldsymbol{V}}
\newcommand{\I}{\boldsymbol{I}}
\newcommand{\T}{\boldsymbol{T}}

\newcommand{\n}{\boldsymbol{n}}

\newcommand{\nablat}{{{\nabla}_{\Gamma}}}

\newcommand{\sH}{{\rm H}}

\renewcommand{\div}{{\operatorname{div}}}

\newcommand{\refd}{{\operatorname{ref}}}

\newcommand{\isdef}{\mathrel{\mathrel{\mathop:}=}}

\newcommand{\bs}[1]{{\boldsymbol#1}}

\begin{document}
\title[Isogeometric shape optimization for scaffold
structures]{Isogeometric shape optimization for scaffold
structures}
\author{Helmut Harbrecht}
\address{Helmut Harbrecht and Remo von Rickenbach,
Departement f\"ur Mathematik und Informatik, 
Universit\"at Basel, 
Spiegelgasse 1, 4051 Basel, Switzerland.}
\email{\{helmut.harbrecht,remo.vonrickenbach\}@unibas.ch}
\author{Michael Multerer}
\address{
Michael Multerer,
Euler Institute,
USI Lugano,
Via la Santa 1, 6900 Lugano, Svizzera.}
\email{michael.multerer@usi.ch}
\author{Remo von Rickenbach}

\begin{abstract}
The development of materials with specific structural properties 
is of huge practical interest, for example, for medical applications 
or for the development of light weight structures in aeronautics. 
In this article, we combine shape optimization and 
homogenization for the optimal design of the microstructure
in scaffolds. Given the current microstructure, we apply the
isogeometric boundary element method to compute the effective 
tensor and to update the microstructure by using the shape gradient 
in order to match the desired effective tensor. Extensive numerical
studies are presented to demonstrate the applicability and feasibility of
the approach.
\end{abstract}

\maketitle

%======================================
\section{Introduction}
%======================================
Many engineering problems amount to 
boundary value problems for an unknown function, which 
needs to be appropriately approximated in 
order to infer some desired quantity of interest. 
If the problem under consideration exhibits different length scales,
homogenization is the method of choice to avoid the expensive numerical 
resolution of the different scales involved.
The idea of homogenization is to replace the 
governing mathematical equations at multiple scales 
by approximate governing equations which only exhibit a 
single scale, see \cite{A,BLP,CIO,T} for example.

Within this approach, the article at hand focuses on
scaffold structures.
Indeed, additive manufacturing allows to build lattices or 
perforated structures and hence to build structures with physical 
properties that vary in space. Assuming a lattice structure of
the material under consideration, one
may compute the effective material tensor on the microscale 
in order to derive an ``effective'' equation on the macroscale.
Taking this consideration as a starting point, we address 
the optimal design of such scaffold structures. This task is 
motivated by the need of scaffold structures which realize
specific properties. The efficient numerical solution of this
problem has, of course, a huge impact on 
many practical applications. Examples range from the development 
of light weight structures in aeronautics to medical implants 
in the orthopedic and dental fields, see e.g.\ \cite{NRD,WXZ} 
and the references therein.

We remark that the optimal design of scaffold structures 
has already been considered in many works, see 
\cite{AOTHH,CCFK,FCHO,HMT,HWGDFS,LKH,LRC,
PFM,RSZ,SSW,S,WK,WWSK} for some of the respective 
results. The methodology used there is mainly based on 
the forward simulation of the material properties of a given 
microstructure. Whereas, sensitivity analysis has been 
used in \cite{AGP,GAP} to compute the derivatives with 
respect to the side lengths and the orientation of rectangular 
cavities. In \cite{HRLS}, the derivatives with respect to the coefficients 
of a B-spline parametrization of the cavity have been computed.
In \cite{NC}, the shape derivative has been derived in the context 
of a level set representation of the inclusion. Rather than pursuing
one of the aforementioned approaches, we follow
\cite{DH,HADA} and discretize Hadamard's shape gradient directly. 
Shape sensitivity analysis is a powerful tool as it gives explicit knowledge
about the dependency of the functional under consideration on shape 
perturbations. As a consequence, we are able to perform a gradient 
based optimization of the scaffold structure.

Shape optimization has been proven an efficient tool for 
designing structures, which should be constructed with respect 
to certain optimal design considerations. Having the shape gradient at hand, it 
can be applied to optimize the structure under consideration 
with respect to the underlying shape functional, see \cite{ZOL1,MS,
PIR,SI,ZOL2} and the references therein for an overview 
on the topic of shape optimization. Note that shape optimization
falls into the general setting of optimal control of partial differential 
equations.

For the numerical realization of our approach, we adopt the
 \emph{isogeometric analysis} (IGA) framework. IGA has been
introduced in  \cite{HCB05} in order to incorporate simulation techniques 
into the design workflow of industrial development and thus allows to deal
with domain deformations in a straightforward manner. Especially, 
by representing the computational geometry and deformation fields by 
\emph{non-uniform rational B-splines} (NURBS), shapes can easily
be deformed by directly updating the NURBS mappings which are used to
represent the shape under consideration. The particular class of shape
deformations considered within this work are the eigenfunctions of a
prescribed covariance kernel. This way, we determine principal
displacement fields that are tailored to the underlying geometry.
In particular, this approach is independent of the
geometry's smoothness or genus, 
as we illustrate within our numerical experiments.

The cell functions, which are required for the computation of
the effective tensor, are determined by means of boundary integral
equations. Especially, we show that all computations can be conducted at
the boundary of the cavity, including the evaluation of the shape functional
and the shape gradient. This especially allows for dealing with large
deformations  without having to update any volume mesh. For numerical 
computations, we rely on the fast isogeometric boundary element 
method, developed in \cite{DHK+20,DHK+18,DHP16,HP13}, which 
is available as \verb|C++|-library \verb+bembel+ \cite{DHK+20,bembel}. 
In order to speed up computations, \verb+bembel+ employs 
$\mathcal{H}^2$-matrices with the interpolation based
fast multipole method \cite{GR87,GR91,HB02}.
 
This article is structured as follows. In Section~\ref{sec:problem},
we briefly recall the fundamentals of homogenization theory and
introduce the problem under consideration. We especially present
the shape gradient which enables us to find the optimal effective tensor 
by means of a gradient based optimization method. The topic of 
Section~\ref{sct:IGABEM} is the reformulation of the cell problem
by means of boundary integral equations. This amounts to the so-called 
Neumann-to-Dirichlet map, where the Green's function for the 
periodic Laplacian is computed along the lines of \cite{BG,CZ}. 
After having solved the 
Neumann-to-Dirichlet map, both, it suffices to evaluate 
the shape functional and the 
shape gradient exclusively at the boundary of 
the cavity. Section~\ref{sct:discretized} is dedicated to the 
discretization of the shape optimization problem, which comprises
the discretization of the shape and the discretization of the 
Neumann-to-Dirichlet map. Namely, we show how to define 
shape deformations by using the eigenfunctions of a prescribed 
covariance kernel. We also summarize the solution of
the Neumann-to-Dirichlet map by the isogeometric boundary
element method. Various numerical experiments are then performed
in Section~\ref{sct:numerix}. We consider simple cavities like
the sphere and the cube, a multiple cavity consisting of a 
sphere and a cube, and finally a drilled cube which is a cavity
of genus 12. Finally, in Section~\ref{sct:conclusion}, we state 
concluding remarks.

%==============================================
\section{Shape optimization for scaffolds}\label{sec:problem}
%==============================================
\subsection{Homogenization}
%==============================================
We start by outlining the
 approach considered in this article. To this end,
we shall restrict ourselves to the
situation of the simple two-scale 
problem posed on a domain $D\subset\mathbb{R}^3$, i.e.\
\begin{equation}\label{homo1}
%=====================================
-\div\big(\boldsymbol{A}^\varepsilon\nabla u^\varepsilon\big) = f\ \text{in $D$},
\quad u^\varepsilon=0\ \text{on $\partial D$}.
\end{equation}
Herein, the $(3\times 3)$-matrix $\boldsymbol{A}^\varepsilon$ is 
assumed to be oscillatory in the sense that
\[
\boldsymbol{A}^\varepsilon(\bx) = \boldsymbol{A}\biggl({\bx\over\varepsilon}\biggr),
	\quad \bx\in D.
\]
Mathematical homogenization is the study of the limit of 
$u^\varepsilon$ when $\varepsilon$ tends to $0$. Various 
approaches have been developed for this purpose. The oldest one 
is comprehensively presented in Bensoussan, Lions and Papanicolaou 
\cite{BLP}. It consists of performing a formal multiscale asymptotic expansion 
and then, in the justification of its convergence, using the energy 
method due to Tartar \cite{T}. A significant result obtained 
from this approach is the existence of the ($L^2(D)$-) limit 
$u_0(\bx)$ of $u^{\varepsilon}(\bx)$ and, more importantly, 
the identification of a limiting, ``effective'' or ``homogenized'' 
elliptic problem in $D$ satisfied by $u_0$.

\begin{figure}[hbt]
\begin{center}
\setlength{\unitlength}{0.8cm} 
\begin{picture}(13,7.2)
\put(0.015,0.02){\includegraphics[width=4.96cm,height=4.96cm]{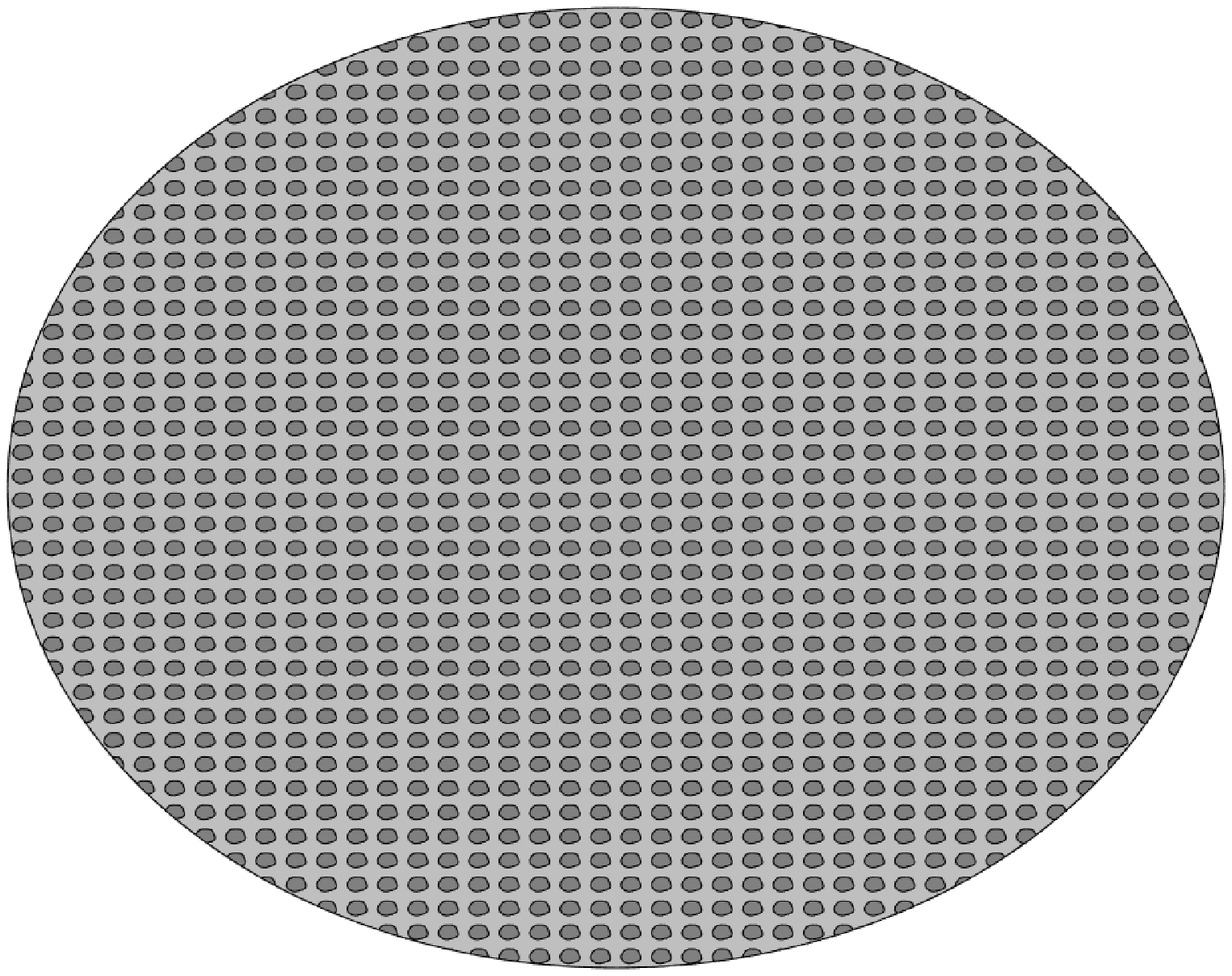}}
\put(2.975,2.975){\line(1,0){0.15}}
\put(2.975,2.975){\line(0,1){0.15}}
\put(2.975,3.125){\line(1,0){0.15}}
\put(3.125,2.975){\line(0,1){0.15}}
\put(2.975,3.125){\line(3,2){4.23}}
\put(2.975,2.975){\line(3,-2){4.23}}
\put(7.2,0.13){\includegraphics[width=4.66cm,height=4.66cm]{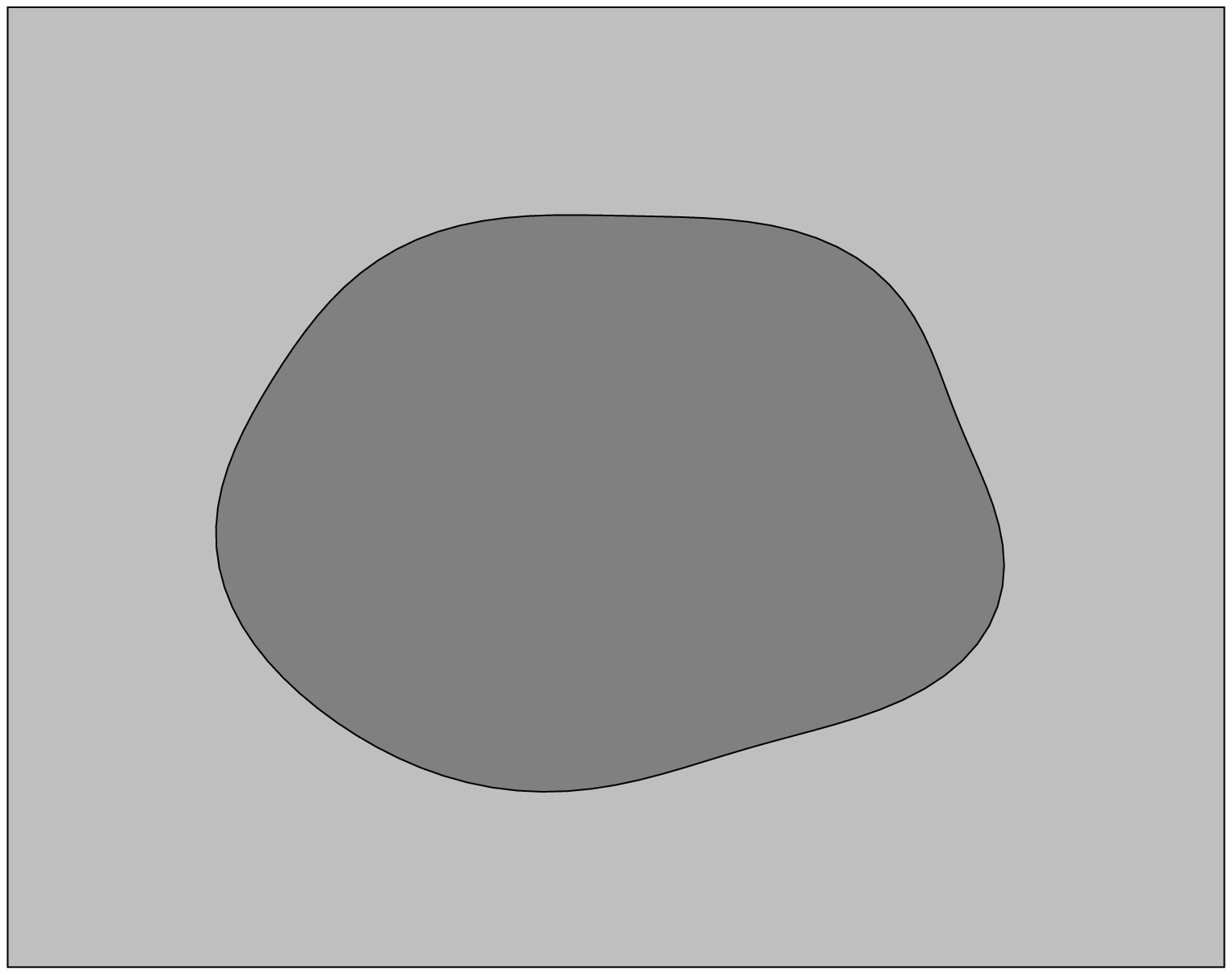}}
\qbezier[1000](3.125,3.125)(7.2,4.3)(7.2,4.3)
\qbezier[1000](3.125,2.975)(7.2,1.8)(7.2,1.8)
\put(7.2,0.13){\line(0,1){5.8}}
\put(7.2,0.13){\line(1,0){5.8}}
\put(13,0.13){\line(0,1){5.8}}
\put(7.2,5.95){\line(1,0){5.8}}
\put(0.5,5.2){$D$}
\put(7.3,5.5){$\partial Y$}
\put(10,3){$\Omega$}
\put(9,4){$\Gamma$}
\put(11.9,4.9){$Y$}
\end{picture}
\caption{
The domain $D$ with unit cell ${Y}$.}
\label{fig:setup}
\end{center}
\end{figure}

We introduce the unit cell $Y=[-1/2, 1/2]^3$ for the fast scale 
of problem \eqref{homo1} and assume that the matrix 
field $\boldsymbol{A}(\by)$ is $Y$-periodic, cf.~Figure
\ref{fig:setup} for a graphical illustration. Moreover, we 
consider the space $\sHoneper(\by)$ of $Y$-periodic 
functions with vanishing mean that belong to $\sH^1(\by)$ 
and the unit vector $\ei \in \R^3$ in the $i$-th direction of 
$\R^3$. We define the cell problems for $i=1,2,3$
according to
\begin{align*}
&\text{find $w_i\in\sHoneper(\by)$ such that}\\
&\qquad-\div\big(\boldsymbol{A}(\by)(\ei+\nabla w_i)\big) = 0.
\end{align*}
The Lax-Milgram lemma ensures the existence and 
uniqueness of the solutions $w_i$ to these cell problems
for $i=1,2,3$.

The family of functions $w_i$ then determines the 
\emph{effective tensor} 
\[\boldsymbol{A}_0 = [a_{i,j}]_{i,j=1}^3
\]
in accordance with
\[
a_{i,j} = \int_Y \langle\boldsymbol{A}(\ei +\nabla w_i),\ej + \nabla w_j\rangle\dd\by.
\]
Based on this tensor, we may compute the
\emph{homogenized solution\/} $u_0\in
\sH_0^1(D)$ by means of the boundary value problem
\[
-\div\big(\boldsymbol{A}_0\nabla u_0\big)  = f\ \text{in}\ D,
\quad u_0 = 0\ \text{on}\ \partial D.
\]
In particular, by setting
\[
  u_1(\bx,\by) = \sum_{i=1}^3 \frac{\partial u_0}{\partial x_i}(\bx)w_i(\by),
  	\quad (\bx,\by)\in D\times Y,
\]
there holds the error estimate
\[
  \bigg\|u^\varepsilon(\bx)-u_0(\bx)-\varepsilon u_1\bigg(\bx,{\bx\over\varepsilon}\bigg)\bigg\|_{H^1(D)}
  	\leq c\sqrt{\varepsilon}\to 0\ \text{as $\varepsilon\to 0$}
\]
 for some constant \(c>0\), cf.~\cite{A,N89}.

%==============================================
\subsection{Scaffold structures}
%==============================================
From now on, we shall consider a scaffold structure.
To this end, we assume that
 the unit cell $Y = [0,1]^3$ is comprised of a 
 homogeneous material which contains some cavity
 $\Omega$. More precisely, let $\Omega$ be an open subset of $Y$ and
set $\Gamma\isdef\partial\Omega$. The collection of interior 
boundaries, being translates of $\varepsilon\Gamma$ of the 
macroscopic domain $D^\varepsilon$, is denoted by $\partial 
D_{\text{int}}^\varepsilon$ while the remainder of the boundary 
$\partial D^\varepsilon\setminus\partial D_{\text{int}}^\varepsilon$ 
is denoted by $\partial D_{\text{ext}}^\varepsilon$. The current 
situation is illustrated in Figure~\ref{fig:setup}.

In accordance with \cite{CIO}, we consider the boundary 
value problem
\begin{equation}\label{homo2}
%=====================================
\begin{aligned}
-\div\big(\boldsymbol{A}^\varepsilon\nabla u^\varepsilon\big) &= f\ &&\text{in $D^\varepsilon$},\\
\langle\boldsymbol{A}^\varepsilon\nabla u^\varepsilon,\n\rangle &=0\ 
&&\text{on $\partial D_{\text{ext}}^\varepsilon$},\\
u^\varepsilon&=0\ &&\text{on $\partial D_{\text{int}}^\varepsilon$}.
\end{aligned}
\end{equation}
Here, we have $\boldsymbol{A}^\varepsilon = 
\boldsymbol{A}(\cdot/\varepsilon)$ in $D^\varepsilon$ with
$\boldsymbol{A} = \boldsymbol{I}$ in $Y\setminus\overline{\Omega}$.
Moreover, the surface $\Gamma$ of the cavity $\Omega$ 
is assumed to be oriented such that its normal vector $\n$ 
indicates the direction going from the interior of $\Omega$ 
to the exterior $Y\setminus\overline{\Omega}$.

To derive the homogenized problem, one introduces the cell functions 
$w_i\in\sHoneper(Y\setminus\overline{\Omega})$ which are now 
given by the Neumann boundary value problems
\begin{equation}\label{cell:problem}
%=====================================
\begin{aligned}
  \Delta w_i &= 0 &&\quad \text{in $Y\setminus\overline{\Omega}$},\\
   \partial_{\n}w_i &= -\langle \n,\ei\rangle &&\quad \text{on $\Gamma$}.
\end{aligned}
\end{equation}
The homogenized equation becomes
\[
-\div\big(\boldsymbol{A}_0(\Omega)\nabla u_0\big)  = (1-|\Omega|) f\ \text{in}\ D,
\quad u_0 = 0\ \text{on}\ \partial D.
\]
Here, the domain $D$ coincides with $D^\varepsilon$ 
except for the holes. The effective tensor $\boldsymbol{A}_0(\Omega) 
= [a_{i,j}(\Omega)]_{i,j=1}^3$ is now given by
\begin{equation}\label{equi:tensor}
%=======================================
a_{i,j}(\Omega) =\int_{Y\setminus\overline{\Omega}}
	\langle\ei +\nabla w_i,\ej + \nabla w_j\rangle\dd\by,
\end{equation}
compare \cite{CIO}. Notice that $\boldsymbol{A}_0(\Omega)$ 
is a symmetric matrix. However, it is not the identity in general, since 
the geometry of the cavity induces a global anisotropy.

%==============================================
\subsection{Shape calculus}
%==============================================
Considering a given tensor $\boldsymbol{B}
\in\mathbb{R}_{\text{sym}}^{3\times 3}$ which describes
the desired 
material properties, we my ask the following question: 
Can we find a cavity, i.e.\ a domain $\Omega$, such 
that the effective tensor is as close as possible to 
$\boldsymbol{B}$? 
 
In order to make the notion of closeness between matrices
precise, we choose the Frobenius norm 
of matrices and define the shape functional $J(\Omega)$ 
according to
\begin{equation}\label{eq:functional}
%===================================
J(\Omega) = \cfrac{1}{2}\|\boldsymbol{A}_0(\Omega)-\boldsymbol{B}\|_{F}^2 
	= \frac{1}{2} \sum_{i,j=1}^3 \big(a_{i,j}(\Omega)-b_{i,j}\big)^2.
\end{equation}
We mention, however, that it cannot be expected that 
every tensor can be matched by a corresponding scaffold structure. 

In order to solve the corresponding minimization problem
\[
J(\Omega)\to\min,
\] we
shall employ shape optimization. We are hence interested in 
describing how the effective tensor depends on a given deformation 
field, which acts on the shape of the cavity $\Omega$. To this end, 
we introduce the displacement field 
$\h\colon Y\to Y$ that vanishes at the boundary $\partial Y$ of the 
reference cell but whose action may deform the interior boundary
$\Gamma$. The deformation field is then a perturbation of identity $\T_t=\I+t\h$, 
which is a diffeomorphism for sufficiently small values $t>0$ and preserves $Y$. 
We denote by $\Omega_t =\T_t(\Omega)$ and by $w_i(t)\in\sHoneper(\by)$ 
the solution of \eqref{cell:problem} for the cavity $\Omega(t)$. Then, 
the shape derivative of a given shape functional $J(\Omega)$ is 
defined according to
\begin{equation}\label{eq:derivative_def}
%====================================
  J'(\Omega)[\h] = \lim_{t\to 0}\frac{J(\Omega_t)-J(\Omega)}{\varepsilon}.
\end{equation}

The shape derivative \eqref{eq:derivative_def}
with respect to the functional \eqref{eq:functional}
has been computed in \cite{DH} and is given by
\begin{equation}\label{eq:gradient}
%===================================
  J'(\Omega)[\h] = \sum_{i,j=1}^3\big(a_{i,j}(\Omega)-b_{i,j}\big)
  	a_{i,j}'(\Omega)[\h],
\end{equation}
where the coefficients $a_{i,j}'(\Omega)[\h]$ are given by
\begin{equation}\label{shape:derivative}
%=================================================
a_{i,j}'(\Omega)[\h] = -\int_{\Gamma}\langle\nablat\phi_i,
 	\nablat\phi_j\rangle\langle\h,\n\rangle\dd o
		\quad \text{with $\phi_i = x_i + w_i$}.
\end{equation}

\section{Boundary integral equation for the cell problem} 
\label{sct:IGABEM}
%==============================================
\subsection{Green's function} 
%==============================================
In order to numerically solve the cell problem \eqref{cell:problem}, 
we shall recast it as a boundary integral equation. To that
end, we have first to determine the $Y$-periodic kernel function
$k_{per}(\bz)$ which satisfies $-\Delta k_{\text{per}}(\bz) = \delta_{\bs 0}
(\bz)$. We follow \cite{BG,CZ} and make the ansatz
\[
  k_{\text{per}}(\bz) =  \frac{1}{4\pi}\sum_{\bm\in\{-1,0,1\}^3}
  	\frac{1}{\|\bz-\bm\|} + \frac{\|\bz\|^2}{6} + k_{\text{corr}}(\bz),
\]
where the correction $k_{\text{corr}}(\bz)$ is given by
\begin{equation}\label{eq:correction}
%==================================
  k_{\text{corr}}(\bz) = \sum_{n=0}^\infty
  \sum_{\ell=-n}^n \alpha_{n,\ell} R_{\ell}^n(\bz).
\end{equation}
Here, $\{R_{\ell}^n\}$ denote the (regular) solid harmonics
and the coefficients $\{\alpha_{n,\ell}\}$ are chosen such that 
$k_{\text{per}}(\bz)$ is $Y$-periodic. In practice, we truncate the 
expansion after $N$ terms and compute the coefficients 
$\alpha_{n,\ell}$ by minimizing the deviation of $k_{\text{per}}$ 
from a $Y$-periodic function. The necessary truncation rank 
to achieve machine precision is rather small. For example,
we use $N = 12$ to obtain an accuracy of about $5\cdot 10^{-8}$.
The coefficients 
$\alpha_{n,\ell}$ can of course be computed in an offline phase
and may then be tabulated for later use.

%==============================================s

\subsection{Neumann-to-Dirchlet map} 
%==============================================
Having the periodic Green's function at hand, we define the corresponding
single layer operator $\mathcal{V}$ in accordance with
\begin{equation}\label{eq:SL}
%=================================
  \mathcal{V}\colon\sH^{-1/2}(\Gamma)\to\sH^{1/2}(\Gamma),\quad
  (\mathcal{V}w)(\bx) = \int_\Gamma k_{\text{per}}(\bx-\by)
  	w(\by)\dd o_{\by}
\end{equation}
and the double layer operator $\mathcal{K}$ 
in accordance with
\begin{equation}\label{eq:DL}
%=================================
  \mathcal{K}\colon\sH^{1/2}(\Gamma)\to\sH^{1/2}(\Gamma),\quad
  (\mathcal{K}w)(\bx) = \int_\Gamma \frac{\partial k_{\text{per}}}{\partial\n_{\by}}(\bx-\by)
  	w(\by)\dd o_{\by}.
\end{equation}
Given the Neumann data $-\langle \n,\ei\rangle$
of the cell function $w_i\in\sHoneper(Y\setminus\overline{\Omega})$ 
at the boundary $\Gamma$, the respective Dirichlet data 
are given by the solution of the boundary integral equation
\begin{equation}\label{eq:N2D}
%=================================
  \bigg(\frac{1}{2}-\mathcal{K}\bigg)w_i = \mathcal{V}\langle \n,\ei\rangle.
\end{equation}
This (exterior) Neumann-to-Dirichlet map is known to be 
uniquely solvable.

%==============================================
\subsection{Evaluating the shape functional and its gradient} 
%==============================================
In this section, we present the formulas to compute the effective 
tensor $\boldsymbol{A}_0(\Omega) = [a_{i,j}(\Omega)]_{i,j=1}^3$ 
given by \eqref{equi:tensor} and its shape derivative
\eqref{shape:derivative} from the Dirichlet and Neumann data
of the the cell function.

\begin{lemma}\label{lem:coefficient}
%===================================
The effective tensor $\boldsymbol{A}_0(\Omega) = 
[a_{i,j}(\Omega)]_{i,j=1}^3$ from \eqref{equi:tensor} 
satisfies the identity
\[
 a_{i,j}(\Omega) = \delta_{i,j} |Y\setminus\overline{\Omega}| - \int_\Gamma w_i \langle\ej,\n\rangle \dd o.
\]
\end{lemma}

\begin{proof}
We will discuss the four terms of the 
effective tensor
\[
a_{i,j}(\Omega) = \int_{Y\setminus\overline{\Omega}}
	\big\{\langle\ei,\ej\rangle + \langle\nabla w_i,\ej\rangle + \langle\ei,\nabla w_j\rangle
	+ \langle\nabla w_i,\nabla w_j\rangle\big\}\dd\by.
\]
separately. For the first term, it holds
\[
 \int_{Y\setminus\overline{\Omega}}
	\langle\ei,\ej\rangle\dd\by = \delta_{i,j} |Y\setminus\overline{\Omega}|.
\]
Then, integration by parts gives
\[
  \int_{Y\setminus\overline{\Omega}}
	\langle\nabla w_i,\nabla w_j\rangle\dd\by 
  = -\int_\Gamma \frac{\partial w_i}{\partial\n} w_j \dd o
  = \int_\Gamma \langle \n,\ei\rangle w_j \dd o,
\]
where we used the $Y$-periodicity of $w_i$ and
$w_j$. Likewise, observing that $\div(\ej)=0$, we find 
for the mixed terms by integration by parts that
\[
  \int_{Y\setminus\overline{\Omega}}
	\langle\nabla w_i,\ej\rangle\dd\by 
= - \int_\Gamma w_i \langle\ej,\n\rangle \dd o.
\]
The assertion is now obtained by adding up the different terms.
\end{proof}

For the shape derivative of the effective tensor,
we find the following expression.

\begin{lemma}\label{lem:dcoefficient}
%===================================
The shape derivative of the effective tensor
 \(\boldsymbol{A}_0(\Omega)\) from \eqref{equi:tensor} is given by
\[
  a_{i,j}'(\Omega)[\h] = \int_{\Gamma} \big\{\langle\ei+\nablat w_i,\ej+\nablat w_j\rangle
 	-\langle\ei,\n\rangle\langle\ej,\n\rangle\big\}\langle\h,\n\rangle\dd o.
\]
\end{lemma}

\begin{proof}
Since $\phi_i = x_i+ w_i$, we find
\[
  a_{i,j}'(\Omega)[\h] = \int_{\Gamma}
     \big\langle \ei-\langle\ei,\n\rangle\n+\nablat w_i,
 	\ej-\langle\ej,\n\rangle\n+ \nablat w_j\big\rangle\langle\h,\n\rangle\dd o.
\]
First, we note that the normal $\n$ and the tangential
gradient $\nabla_\Gamma w_i$ are perpendicular to 
each other, i.e.\ $\langle \nabla_\Gamma w_i,\n\rangle 
= \langle\n,\nablat w_j\rangle = 0$. Moreover, we conclude
\[
\big\langle \ei-\langle\ei,\n\rangle\n,\ej-\langle\ej,\n\rangle\n\big\rangle
  = \langle \ei,\ej\rangle - \langle\ei,\n\rangle\langle\ej,\n\rangle.
\]
Therefore, we get the desired result.
\end{proof}

%==============================================
\section{Discretization}\label{sct:discretized}
%==============================================
\subsection{Parametric representation of the reference shape} 
%==============================================
Our goal is to compute the deformation of a given reference
boundary $\Gamma_\refd$
such that the effective tensor $\boldsymbol{A}_0(\Omega)$
corresponds to the desired one $\boldsymbol{B}$. In the following, 
we will assume the usual isogeometric setting for the reference 
boundary $\Gamma_\refd$. To this end, let the unit square be 
denoted by $\square \isdef [0,1]^2$ and assume that the reference 
boundary $\Gamma_\refd$ is represented by a regular and 
non-overlapping decomposition into smooth, quadrangular 
\emph{patches}
\[
\Gamma_\refd = \bigcup_{i=1}^M \Gamma_{\refd}^{(i)}.
\]
Here, regular and non-overlapping means that the intersection 
$\Gamma_{\refd}^{(i)}\cap \Gamma_{\refd}^{(i')}$ of two patches 
consists at most of a common vertex or a common edge for \(i\neq 
i^\prime\). 

Following the paradigm of isogeometric analysis, each patch 
$\Gamma_{\refd}^{(i)}$ is represented by an 
invertible spline or NURBS mapping
\begin{equation}\label{eq:parametrization}
%==============================================
{\bs s}_i\colon\square\to\Gamma_{\refd}^{(i)}
\quad\text{with}\quad\Gamma_{\refd}^{(i)} = {\bs s}_i(\square)
\quad\text{for}\ i = 1,2,\ldots,M.
\end{equation}
We especially follow the common convention that
parametrizations with a common edge coincide except 
for orientation.

\begin{figure}[htb]
	\begin{center}
		\begin{tikzpicture}[
		scale=.3,
		axis/.style={thick, ->, >=stealth'},
		important line/.style={thick},
		every node/.style={color=black}
		]
		\draw (-.5,-1)node{{$0$}};
		\draw (8,-1)node{{$1$}};
		\draw (-.5,8)node{{$1$}};
		\draw (0,0)--(8,0);
		\draw (0,1)--(8,1);
		\draw (0,2)--(8,2);
		\draw (0,3)--(8,3);
		\draw (0,4)--(8,4);
		\draw (0,5)--(8,5);
		\draw (0,6)--(8,6);
		\draw (0,7)--(8,7);
		\draw (0,8)--(8,8);
		\draw (0,0)--(0,8);
		\draw (1,0)--(1,8);
		\draw (2,0)--(2,8);
		\draw (3,0)--(3,8);
		\draw (4,0)--(4,8);
		\draw (5,0)--(5,8);
		\draw (6,0)--(6,8);
		\draw (7,0)--(7,8);
		\draw (8,0)--(8,8);
		\draw (3,3)node(N1){};
		\draw (20.7,3)node(N2){};
		\draw
		(22.5,4)node{\includegraphics[scale=0.37,clip=true,trim=0 0 0 0]{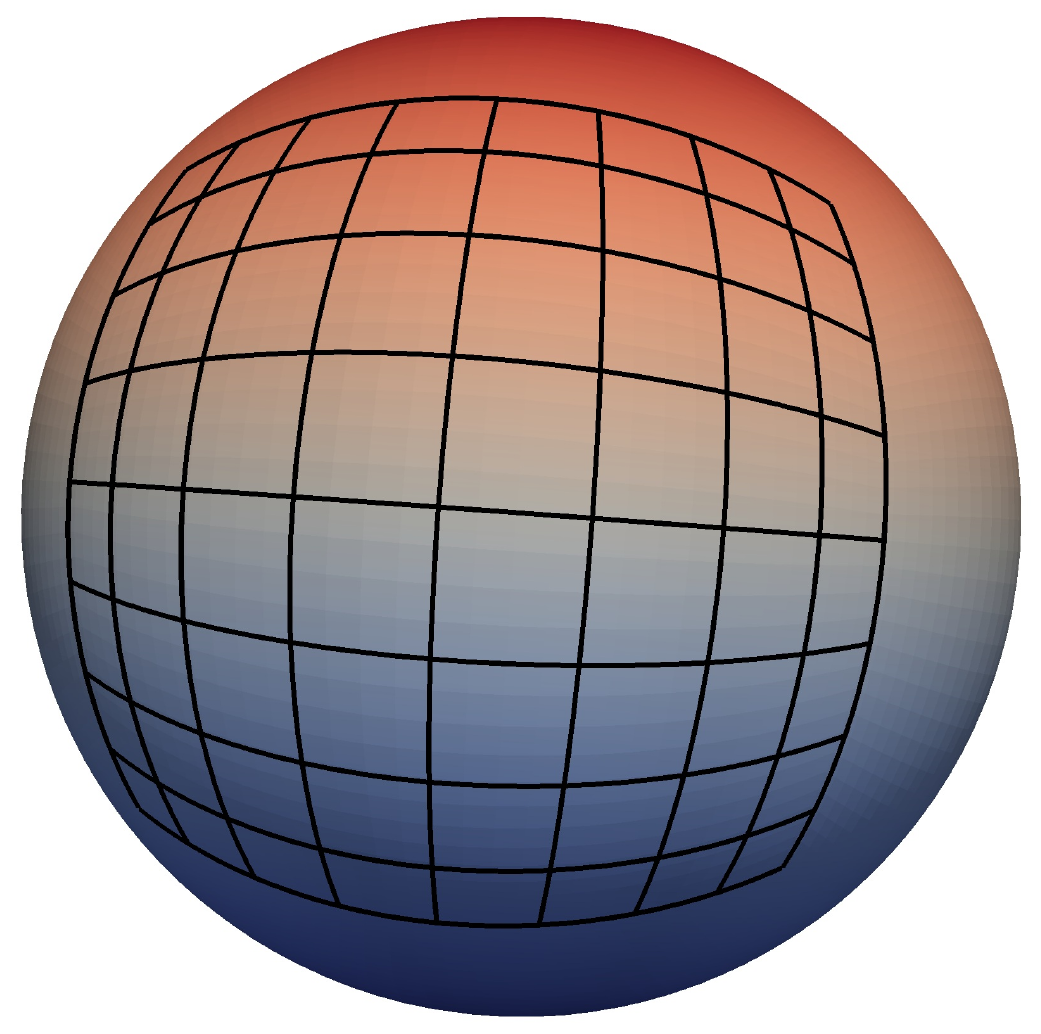}};
		\draw (11.7,6.4)node(N3){${\bs s}_i$};
		\draw (20.6,5.1)node(N4){{\small${\Gamma}_\refd^{(i)}$}};
		\path[->,line width=1pt]  (N1) edge [bend left] (N2);
		\end{tikzpicture}
	\end{center}
	\caption{Surface representation and mesh generation.}
	\label{fig:mesh}
\end{figure}
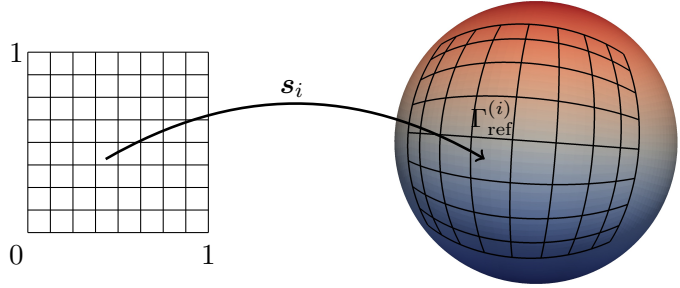

A mesh of level $j$ on $\Gamma$ is finally
induced by dyadic subdivisions of depth $j$ of the unit 
square into $4^j$ squares. We denote these squares by
\[
\square_{k,k'}^{(j)}\isdef[2^{-j}k,2^{-j}(k+1)]
\times [2^{-j}k',2^{-j}(k'+1)].
\]
   This generates $4^jM$ 
{\itshape elements\/} (or elementary domains), 
compare Figure~\ref{fig:mesh}.

%==============================================
\subsection{Shape discretization} 
%==============================================
By assumption, the NURBS patches under consideration
are smooth. As a consequence, 
we may reinterpolate them by using piecewise polynomials
in order to speed up computations in the optimization and
update procedures.
More precisely, given 
the vertices of the subdivided unit square
\[
{\bs\xi}_{\ell,\ell'}\isdef 2^{-j}
\begin{bmatrix}\ell\\\ell'\end{bmatrix}
,\quad \ell,\ell'=0,1,
\ldots 2^j,
\]
we represent the element
\[
\Gamma_{i,j,k,k'}\isdef{\bs s}_i\big(\square_{k,k'}^{(j)}\big)
\]
by piecewise interpolating the parametrization \({\bs s}_i\) 
by tensor product polynomials of degree \({\bs p}=(p_1,p_2)\). 
To that end, we introduce the vertices
\[
{\bs s}_i({\bs\xi}_{\ell,\ell'})\quad\text{for }
\ell=m,m+1,m+p_1,\ 
\ell'=m',m'+1,\ldots,m'+p_2,
\]
where \(m\isdef\min\{k,2^{j}-p_1\}\), \(m'\isdef\min\{k',2^{j}-p_2\}
\), and the Lagrange polynomials 
\(L_0^{(1)}(x),L_1^{(1)}(x),\ldots,L_{p_1}^{(1)}(x)\)
with respect to the abscissae
\(x_\ell\isdef 2^{-j}(m+\ell)\) and 
\(L_0^{(2)}(y),L_1^{(2)}(y),\ldots,L_{p_2}^{(2)}(y)\)
with respect to
\(y_\ell\isdef 2^{-j}(m'+\ell)\). Now, we approximate
\begin{equation}\label{eq:polyApprox}
{\bs s}_i|_{\square_{k,k'}^{(j)}}
\approx\bigg(\sum_{\ell=m}^{m+p_1}\sum_{\ell'=m'}^{m'+p_2}
{\bs s}_i({\bs\xi}_{\ell,\ell'})L_{\ell}^{(1)}\otimes
L_{\ell}^{(2)}\bigg)\bigg|_{\square_{k,k'}^{(j)}}.
\end{equation}
Here, the polynomial degree \({\bs p}\) is always chosen
such that the overall consistency error is met.

We remark that the polynomial approximation \eqref{eq:polyApprox}
particularly allows for the rapid evaluation of geometric quantities, 
such as the surface measure and the normal vector by means of
the complete 
Horner scheme if the Newton basis is used for representing the
polynomials.

For the representation of shape variations, we adopt an
approach that is well established in the context of random
domain variations: we represent these variations by means
of the Karhunen-Lo\`eve expansion, see \cite{GS,L},
of a parametric deformation field with design parameters 
$\{y_k\}_k$, that is
\begin{align}\label{eq:boundarykl}
%===========================================
{\bs\chi}({\bs x},{\bs y})={\bs x}+
\sum_{k=1}^\infty\sqrt{\lambda}_k{\bs V}_{\!k}({\bs x})y_k,
\quad{\bs x}\in\Gamma_\refd.
\end{align}
Herein, \(\{\lambda_k,{\bs V}_{\!k}\}_k\) are the eigenpairs of the
Hilbert-Schmidt operator associated to some matrix valued covariance
function
\begin{equation}\label{eq:covariance}
%===========================================
\operatorname{Cov}[{\bs\chi}]\colon\Gamma_\refd\times\Gamma_\refd
\to\mathbb{R}^{d\times d}.
\end{equation}
To represent the covariance structure of the
interpolation points 
\[
{\bs\Xi}\isdef\big\{{\bs s}_i({\bs\xi}_{\ell,\ell'}):
i=1,2,\ldots,M,\ell,\ell'=0,1,\ldots 2^j\big\},
\]
it is sufficient to consider the matrix 
\[
{\bs C}\isdef
\big[\operatorname{Cov}[{\bs\chi}]({\bs\xi},{\bs\xi}')\big],\quad
{\bs\xi},{\bs\xi}'\in{\bs\Xi}.
\]

Since \({\bs C}\in\mathbb{R}^{3M(2^j+1)^2\times 3M(2^j+1)^2}\), the
solution of the corresponding eigenvalue problem can easily become
prohibitive. Therefore, we shall rely on an approach based on the 
pivoted Cholesky decomposition, compare \cite{HPS12,HPS14}.
Assuming that \({\bs C}\approx{\bs L}{\bs L}^\intercal\) with 
\(\operatorname{rank}{\bs L}=p\ll 3M(2^j+1)^2\), we can replace the 
eigenvalue problem \({\bs L}{\bs L}^\intercal{\bs v}=\lambda{\bs v}\)
by the much smaller one \({\bs L}^\intercal{\bs L}\tilde{\bs v}=\lambda\tilde{\bs v}\).
Having computed an eigenpair \((\lambda_,\tilde{\bs v}_i)\) of the latter,
then \((\lambda_k,{\bs v}_k)\) with ${\bs v}_k\isdef {\bs L}\tilde{\bs v}_k$ 
is an eigenpair of the original eigenvalue problem. Especially, 
there holds \({\bs v}^\intercal_k{\bs v}_k=\lambda_k\). 

As we consider a shape optimization problem, we will drop the 
weighting induced by the eigenvalues \(\{\lambda_k\}_k\) and consider 
only the eigenvectors $\{{\bs v}_k\}_k$ as principal directions for the
domain deformation. Then, by applying 
the interpolation procedure described earlier, we arrive at the parametric 
surface representation
\begin{equation}\label{eq:boundarRep}
%=====================================
\Gamma({\bs y})\isdef{\bs\chi}(\Gamma_\refd,{\bs y})=\Gamma_\refd+
\sum_{k=1}^p{\bs V}_{\!k}({\bs x})y_k,
\quad{\bs y}\in\mathbb{R}^p,
\end{equation}
where the displacement fields \({\bs V}_{\!k}\) are obtained by
polynomial interpolation of the discrete points encoded by
\({\bs v}_k\).

%==============================================
\subsection{Boundary element method} 
%==============================================
We discretize the Neumann-to-Dirichlet map by means of
an isogeometric Galerkin discretization. To this end, let 
$\{\varphi_{j,k}\}$ denote a suitable B-spline basis of order
$d$ on the mesh of level $j$ and define the system matrices
\[
  {\bs S}_j = \bigg[\int_\Gamma (\mathcal{V}\varphi_{j,k'})(\bx)\varphi_{j,k}(\bx)\dd\sigma\bigg]_{k,k'},\quad
  {\bs K}_j = \bigg[\int_\Gamma (\mathcal{K}\varphi_{j,k'})(\bx)\varphi_{j,k}(\bx)\dd\sigma\bigg]_{k,k'}.
\]
They correspond to the single layer operator \eqref{eq:SL}
and double layer operator \eqref{eq:DL}, respectively. 
Moreover, we require the mass matrix
and the right-hand sides
\[
  {\bs M}_j = \bigg[\int_\Gamma \varphi_{k}(\bx)\varphi_{j,k'}(\bx)\dd\sigma\bigg]_{k,k'},
  \quad {\bs b}_{i,j} = \bigg[\int_\Gamma\langle\ei(\bx),\n(\bx)\rangle\varphi_{j,k}(\bx)\dd\sigma\bigg]_k,
\]
where $i=1,2,3$. Then, the discrete version of the 
Neumann-to-Dirichlet map \eqref{eq:N2D} for the computation
of the cell 
functions $w_i\approx \sum_k [{\bs w}_{i,j}]_k\varphi_{j,k}$ is 
given by the linear system of equations
\begin{equation}\label{eq:discreteN2D}
%=============================================
  \bigg(\frac{1}{2}{\bs M}_j+{\bs K}_j\bigg){\bs w}_{i,j} = {\bs S}_j{\bs M}_j^{-1}{\bs b}_{i,j}.
\end{equation}
In practice, we use globally continuous B-spline functions
for the Dirichlet data and patchwise continuous B-spline functions
for the Neumann data, as the Neumann data cannot be expected
to be continuous across patch boundaries if the geometry under
consideration exhibits corners and
edges. 

The implementation of the discretized Neumann-to-Dirichlet 
map \eqref{eq:discreteN2D} is done with the help of the 
\verb|C++|-library \verb+bembel+, where the order of the 
B-splines can be chosen arbitrarily. We refer the reader 
to \cite{DHK+20,bembel} for the details concerning
\verb+bembel+.

Having the Neumann and Dirichlet data of the cell functions at
hand, we are able to compute the shape functional \eqref{eq:functional}
and the shape gradient \eqref{eq:gradient} in accordance with Lemmata 
\ref{lem:coefficient} and \ref{lem:dcoefficient}.

%==============================================
\section{Numerical results}\label{sct:numerix}
%==============================================
In this section, we present extensive
numerical results to demonstrate the feasibility 
of our approach. The following setup is chosen: The 
covariance kernel \eqref{eq:covariance} for generating the displacement 
fields ${\bs V}_{\!k}$ in \eqref{eq:boundarykl} is chosen diagonally 
with diagonal entries given by the Mat\`ern kernel with smoothness 
index $\nu = 9/2$, that is
\begin{equation}\label{eq:matern}
%======================================
  k_{9/2}(r)=\bigg(1+\frac{3r}{\ell}+\frac{27r^2}{7\ell^2}
		+\frac{18r^3}{7\ell^3}+\frac{27r^4}{35\ell^3}\bigg)\exp\bigg(-\frac{3r}{\ell}\bigg),
\end{equation}
where the radius is given by $r = \|\bx-\by\|_2$ and $\ell$
denotes the correlation length. Notice 
that the Mat\`ern kernels $k_\nu$ are the reproducing
kernels of the Sobolev spaces ${\mathrm{H}}^{\nu+3/2}(\mathbb{R}^3)$.
The correlation length is always set to $\ell=1$, unless otherwise 
stated.

In the subsequent examples, we consider a sphere, a cube
 and a drilled cube as initial shapes.
The shape discretization is based on $p = 16$ displacement 
vectors except for the drilled cube, see Section~\ref{sct:drilled}.
The iterative optimization of the initial shape $\Omega$ 
is performed by using the gradient descent method with
quadratic line search, where convergence was always 
obtained within 25 steps. The boundary element 
method for solving the state equation employs quadratic 
B-spline functions on refinement level $j=4$, which amounts 
to about 1700 boundary elements on the sphere and also 
on the cube. In case of the drilled cube, we use refinement 
level $j=3$, which yields 4000 boundary elements.

As the shapes are three-dimensional, we display them with respect
to different directions. The first image is always oriented such
that we are 
looking into $x$-direction, the second one such that we are 
looking into $y$-direction and the third one such that we are 
looking into $z$-direction, see Figure \ref{fig_coordinates}
for the orientation axes.

\begin{figure}[h]
\centering
\begin{tikzpicture}
\draw[->] (0, 0, 0) -- (-2, 0, 0) node[left] {$y$};
\draw[->] (0, 0, 0) -- (0, 2, 0) node[above] {$z$};
\draw[->] (0, 0, 0) -- (0, 0, -2) node[above right] {$x$};
\draw[->] (2, 0, 0) -- (4, 0, 0) node[right] {$x$};
\draw[->] (2, 0, 0) -- (2, 2, 0) node[above] {$z$};
\draw[->] (2, 0, 0) -- (2, 0, -2) node[above right] {$y$};

\draw[->] (7.5, 0, 0) -- (5.5, 0, 0) node[left] {$x$};
\draw[->] (7.5, 0, 0) -- (7.5, 2, 0) node[above] {$y$};
\draw[->] (7.5, 0, 0) -- (7.5, 0, -2) node[above right] {$z$};
\end{tikzpicture}
\caption{The coordinate axes for the plots of the resulting geometries.}
\label{fig_coordinates}
\end{figure}
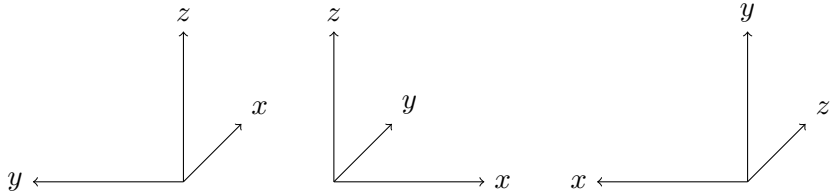

The shape functional assumes very small values in 
each of the numerical examples, which indicates that the solution is not unique
and depends on the particular choice of the displacement vectors.
Indeed, it has already been observed in \cite{DH} that 
the solution of the shape optimization problem under 
consideration depends on the choice of the initial shape. 
This is also observed in our numerical experiments, as 
we present different optimal shapes for the same 
desired effective tensors. 

%==============================================
\subsection{Sphere}
%==============================================
We start with the canonical example of the sphere. To this end, 
let us denote an open ball with radius $r$, centred around the 
point $\bx\in\mathbb{R}^3$, by $B_r(\bx)=\{\by\in\mathbb{R}^3:
\|\by-\bx\|_2<r\}$. Then, in the following examples, we use the 
ball $B_{0.3}({\bs 0})$ as the initial shape and the optimization 
is stopped if $J(\Omega) < 10^{-5}$. For the sphere as an 
initial guess, this criterion was satisfied in the subsequent 
experiments after 3, 6, 17, and 25 iterations, respectively.

In our first test, we used \[
\boldsymbol{B}_1 = 0.9\boldsymbol{I}\quad
\text{and}\quad\boldsymbol{B}_2 = 0.6\boldsymbol{I}\] as desired effective 
tensor $\boldsymbol{B}$ in \eqref{eq:functional}. This choice yields 
just a scaling of the radius of the ball, compare Figure \ref{result_2I} 
for $\boldsymbol{B}_1$ and 
Figure \ref{result_halfI} for $\boldsymbol{B}_2$. 

\begin{figure}
\begin{subfigure}{0.32\textwidth}
\includegraphics[scale=0.19]{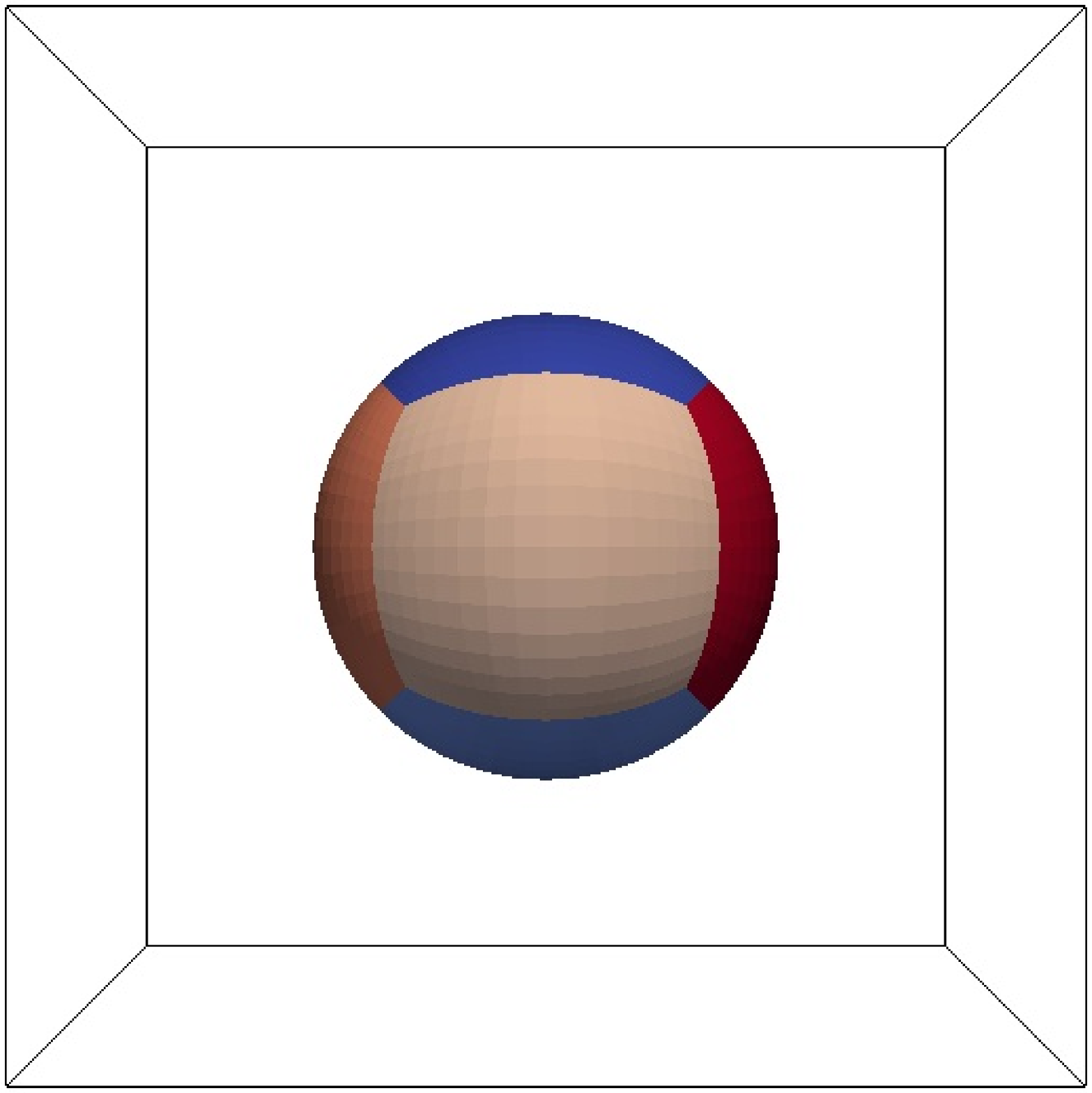}
\end{subfigure}
\begin{subfigure}{0.32\textwidth}
\includegraphics[scale=0.19]{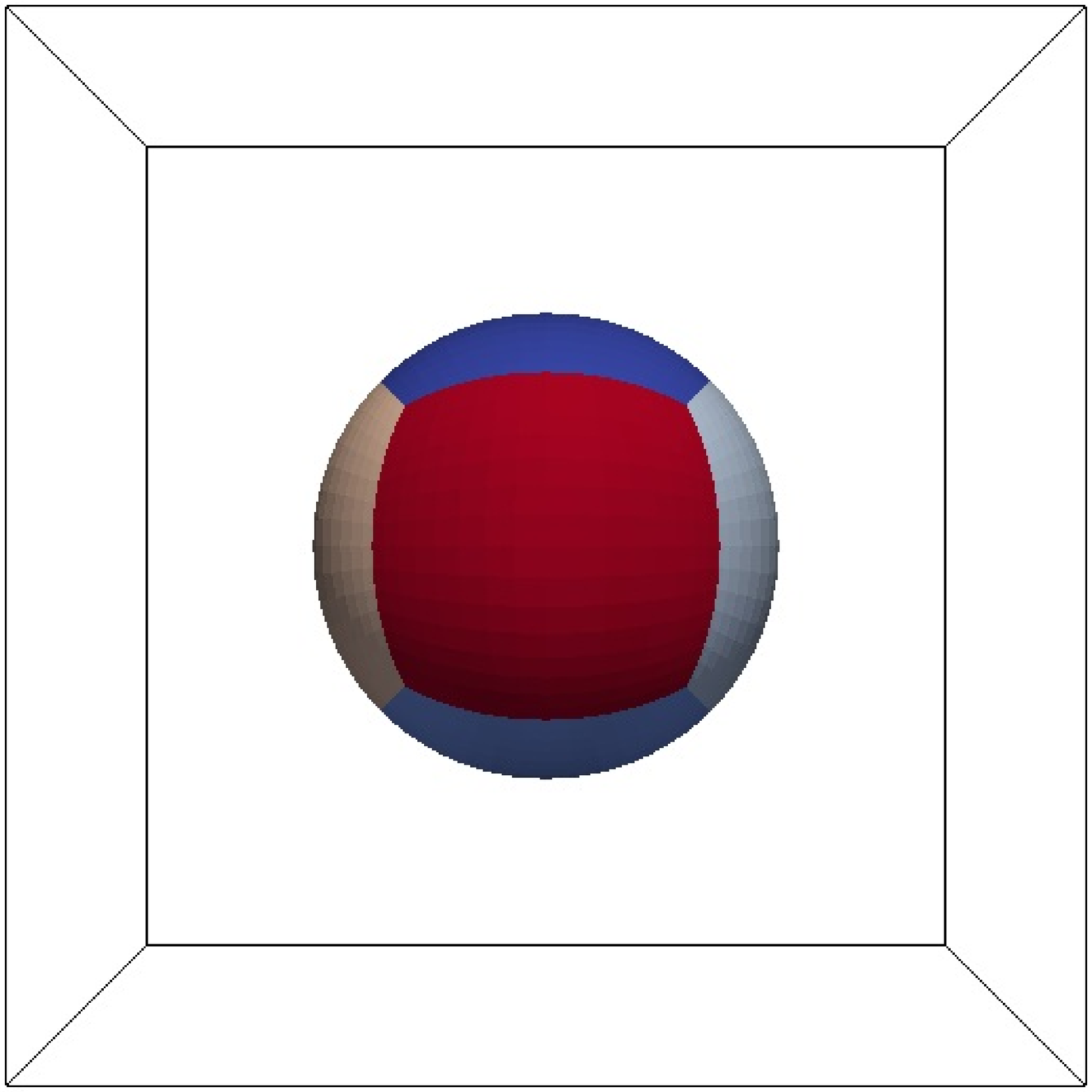}
\end{subfigure}
\begin{subfigure}{0.32\textwidth}
\includegraphics[scale=0.19]{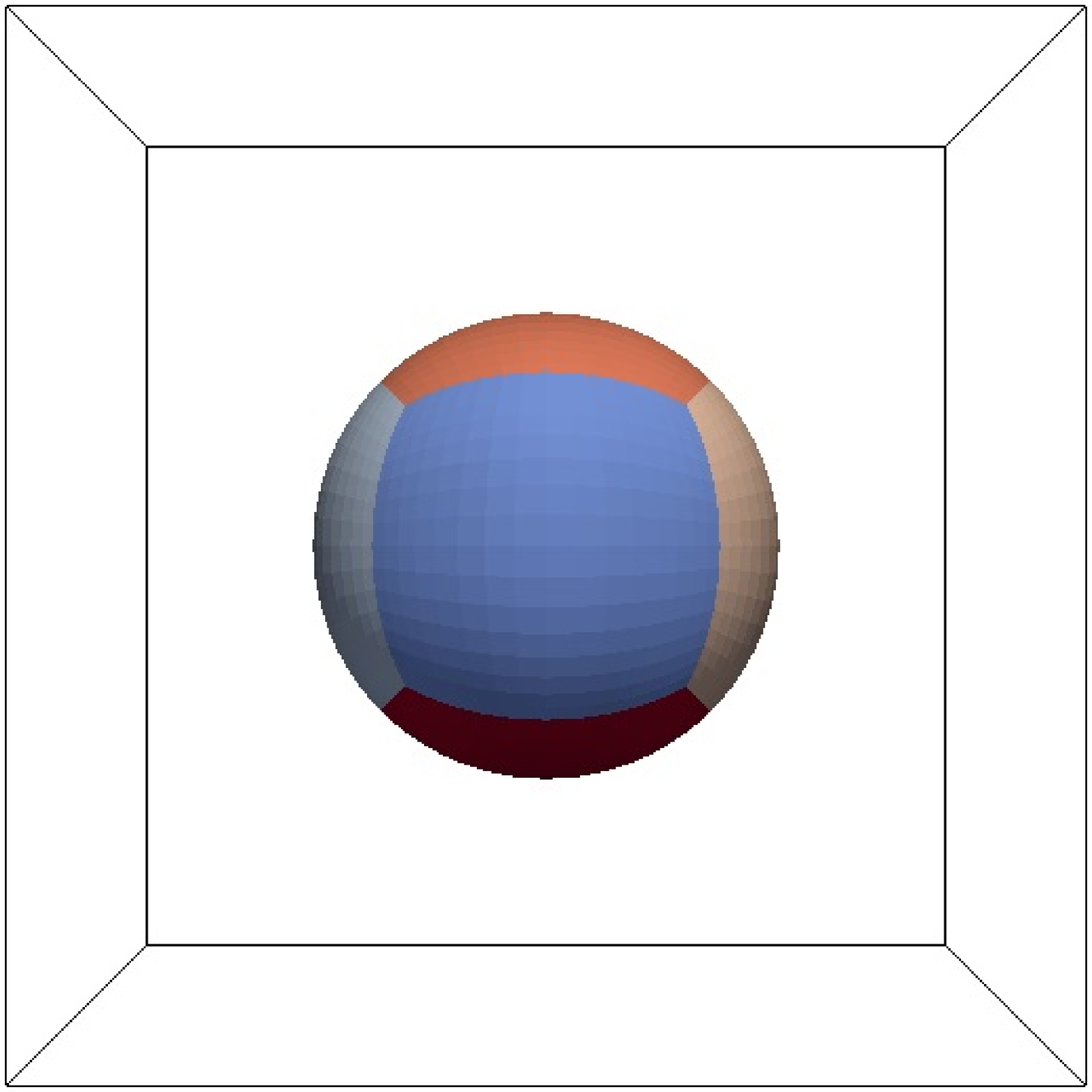}
\end{subfigure}
\caption{Desired tensor $\boldsymbol{B}_1$: Starting 
from the initial guess $B_{0.3}({\bs 0})$, one obtains a ball
of approximate radius $0.25$.}
\label{result_2I}
\end{figure}

\begin{figure}[hbt]
\begin{subfigure}{0.32\textwidth}
\includegraphics[scale=0.19]{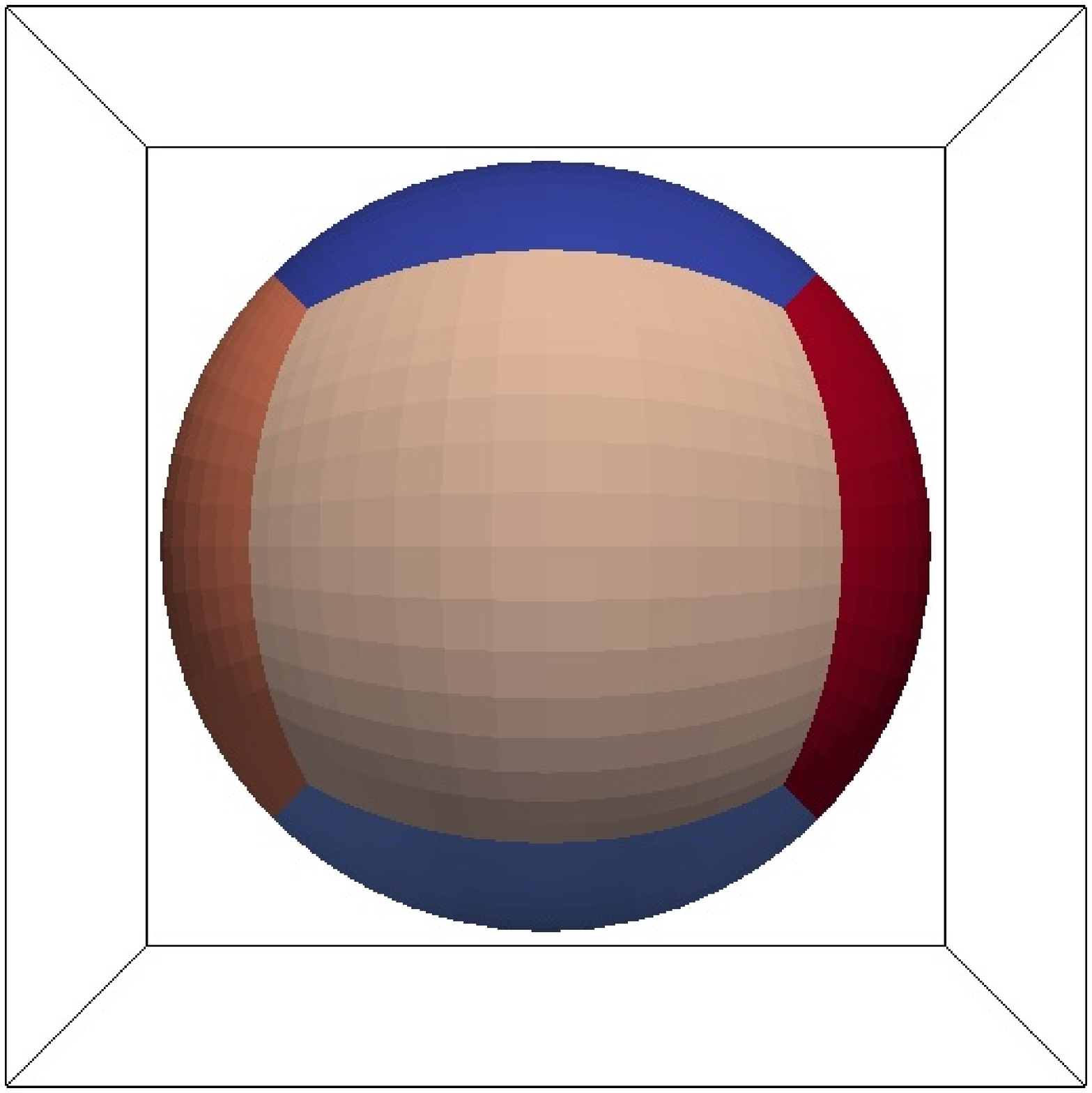}
\end{subfigure}
\begin{subfigure}{0.32\textwidth}
\includegraphics[scale=0.19]{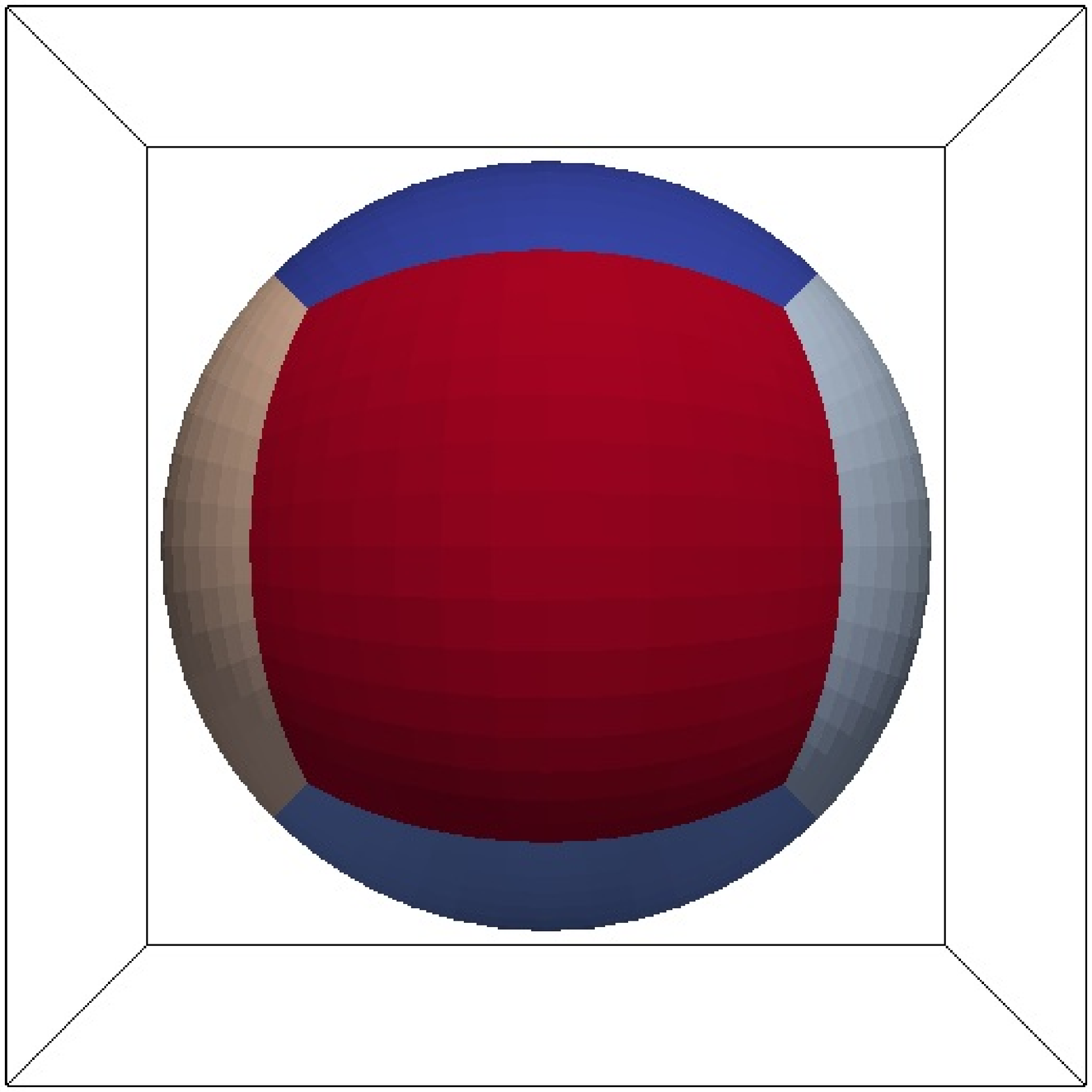}
\end{subfigure}
\begin{subfigure}{0.32\textwidth}
\includegraphics[scale=0.19]{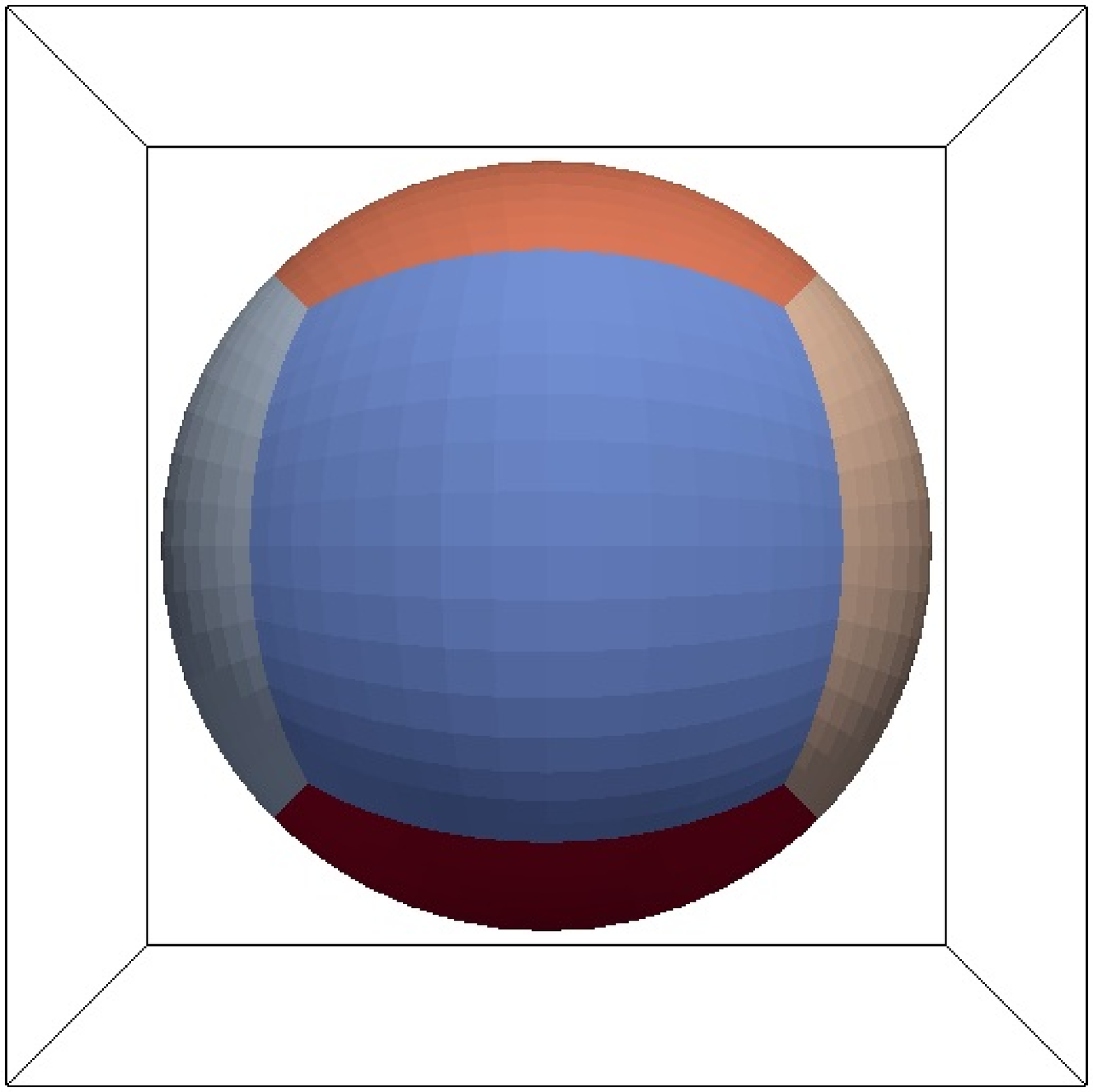}
\end{subfigure}
\caption{Desired tensor $\boldsymbol{B}_2$: 
Starting with the initial guess $B_{0.3}({\bs 0})$, one obtains 
a ball of approximate radius $0.42$.}
\label{result_halfI}
\end{figure}

To obtain an ellipsoid, we prescribe an anisotropic 
effective tensor $\boldsymbol{B}$ in \eqref{eq:functional}. 
Let us therefore define the matrix
\[
\boldsymbol{B}_3 = 
\begin{bmatrix}0.9 & & \\ & 0.88 & \\ & & 0.86\end{bmatrix}.
\] Using $\boldsymbol{B}_3$ 
as the desired effective tensor, the shape of the cavity changes 
significantly, as can be seen in Figure \ref{result_anisotropic}.

In order to account for shapes which are not oriented along the 
coordinate axes, we define the orthogonal transformation
\begin{equation}\label{T_defin}
%======================================
\boldsymbol{T} = \begin{bmatrix}
\frac{1}{\sqrt{3}} & 0 & \frac{2}{\sqrt{6}}\\
\frac{1}{\sqrt{3}} & - \frac{1}{\sqrt{2}} & -\frac{1}{\sqrt{6}}\\
\frac{1}{\sqrt{3}} & \frac{1}{\sqrt{2}} & -\frac{1}{\sqrt{6}}
\end{bmatrix}
\end{equation}
and choose the desired effective tensor \[
\boldsymbol{B}_4 = 
\boldsymbol{TB}_3\boldsymbol{T}^\intercal,\] i.e.\ a rotated 
version of the matrix $\boldsymbol{B}_3$. As can be seen 
in Figure \ref{result_rotated}, the perforation orients differently. 
Indeed, it is now aligned to the rotated coordinate system 
induced by the transformation $\boldsymbol{T}$.

\begin{figure}[hbt]
\begin{subfigure}{0.32\textwidth}
\includegraphics[scale=0.19]{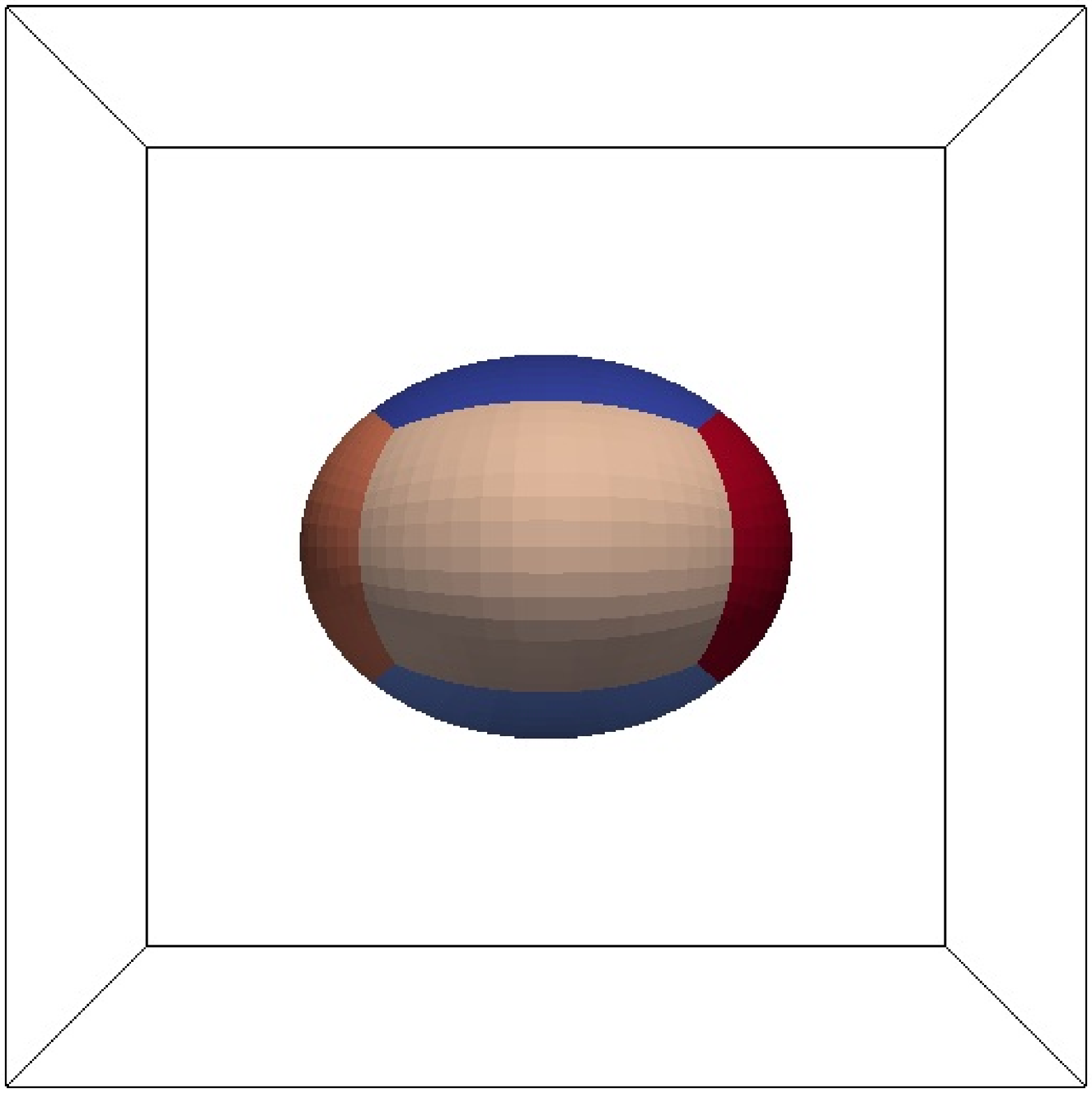}
\end{subfigure}
\begin{subfigure}{0.32\textwidth}
\includegraphics[scale=0.19]{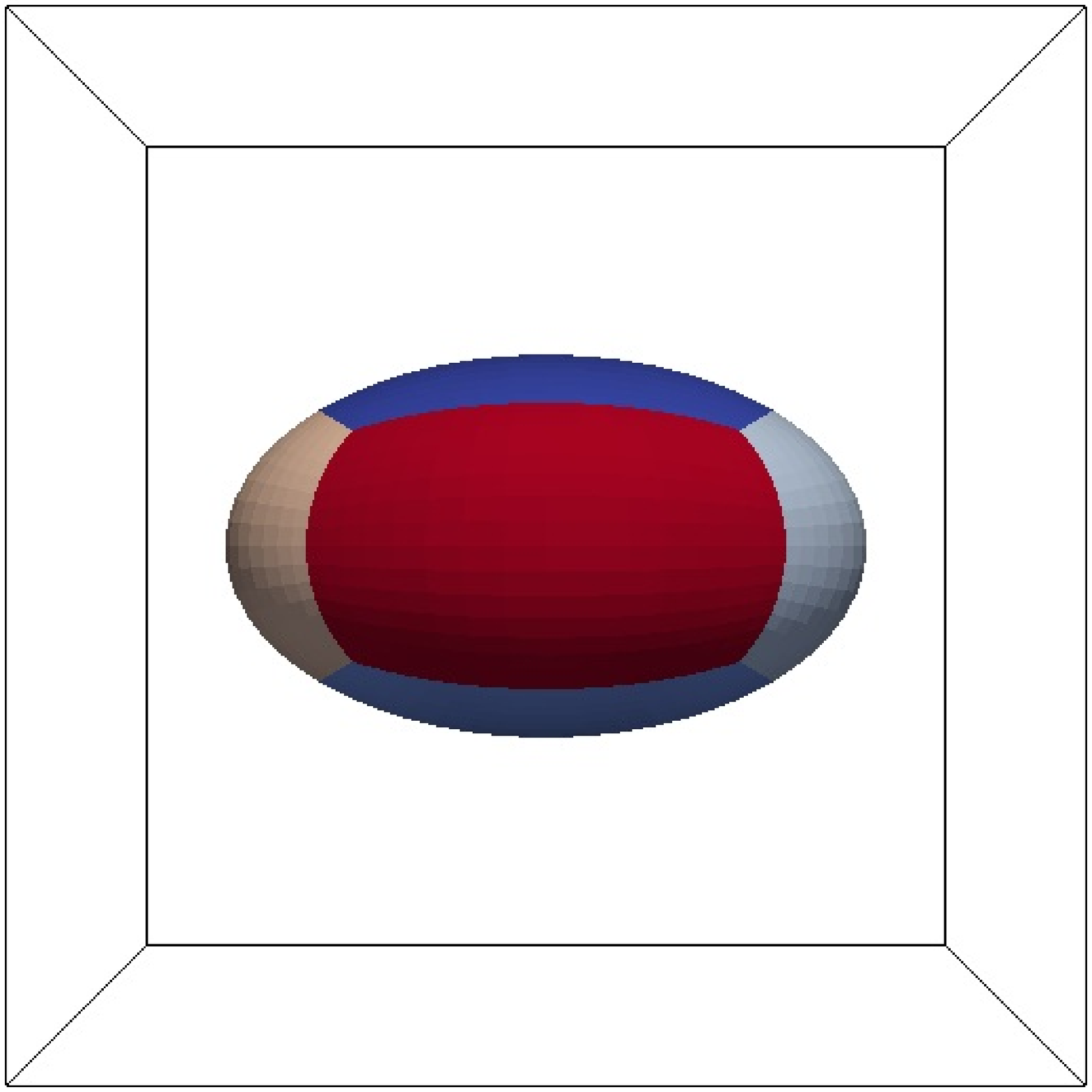}
\end{subfigure}
\begin{subfigure}{0.32\textwidth}
\includegraphics[scale=0.19]{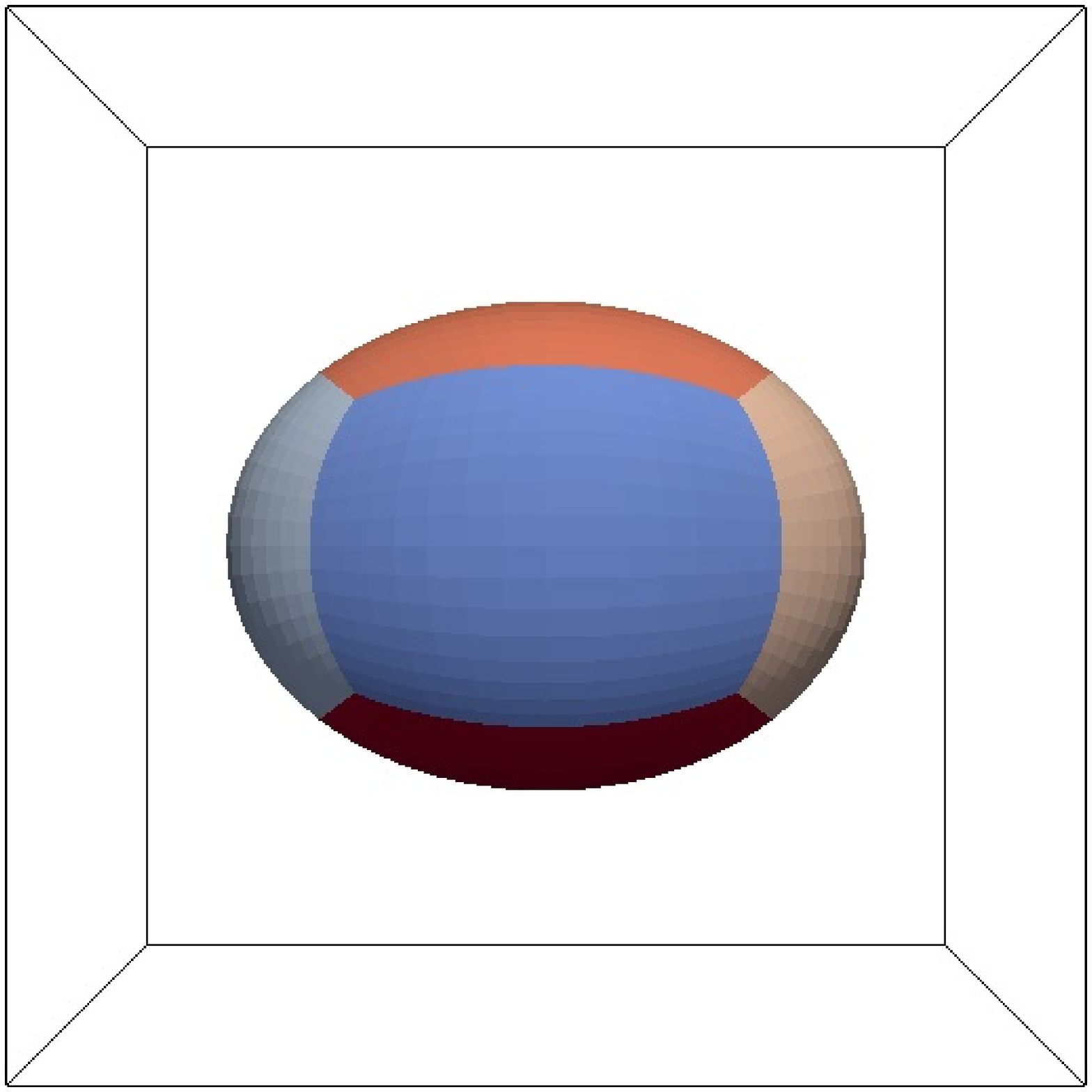}
\end{subfigure}
\caption{Resulting shape in case of the initial guess is 
$B_{0.3}({\bs 0})$ and the desired tensor $\boldsymbol{B}_3$.}
\label{result_anisotropic}
\end{figure}
\begin{figure}[hbt]
\begin{subfigure}{0.32\textwidth}
\includegraphics[scale=0.19]{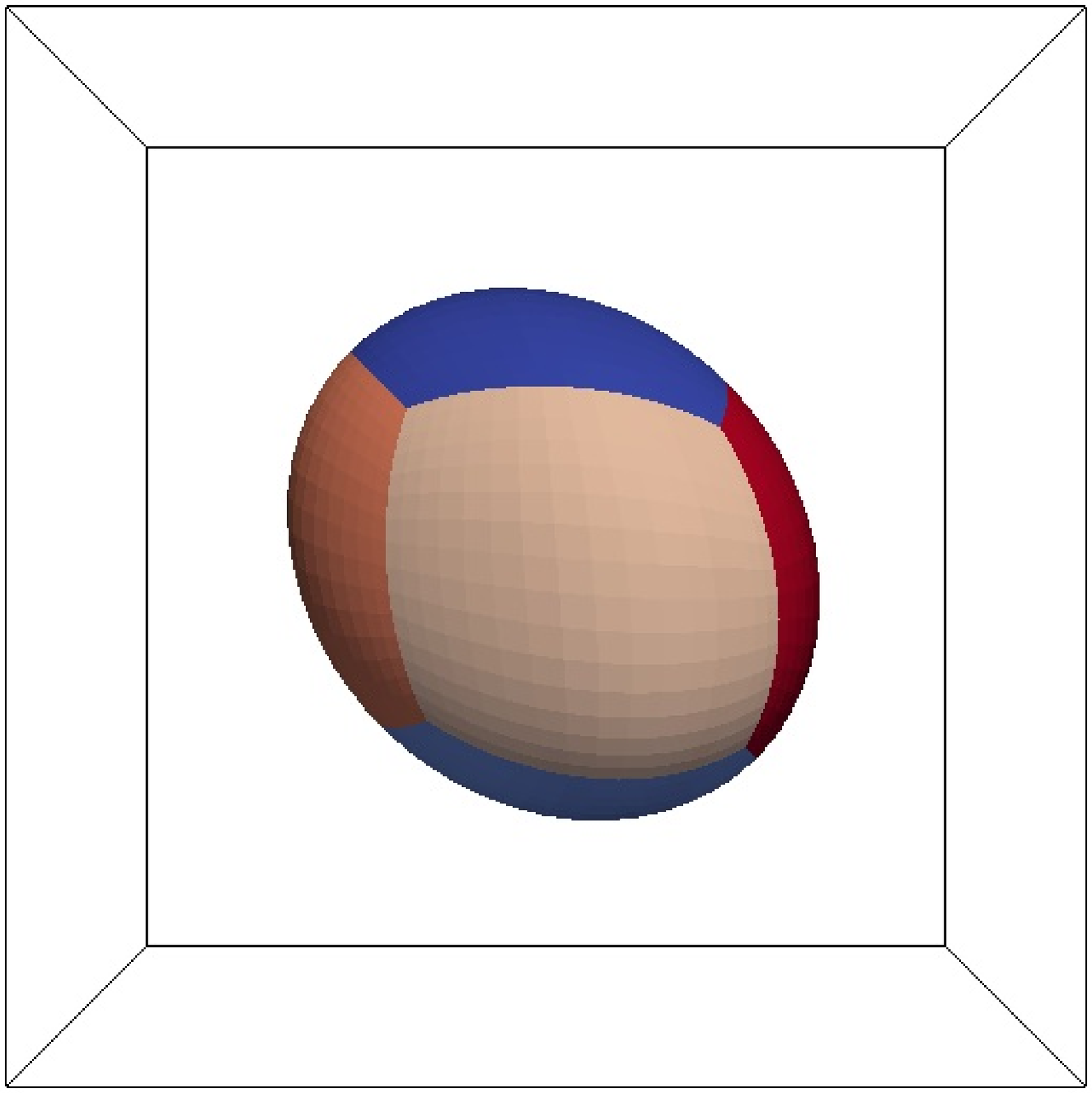}
\end{subfigure}
\begin{subfigure}{0.32\textwidth}
\includegraphics[scale=0.19]{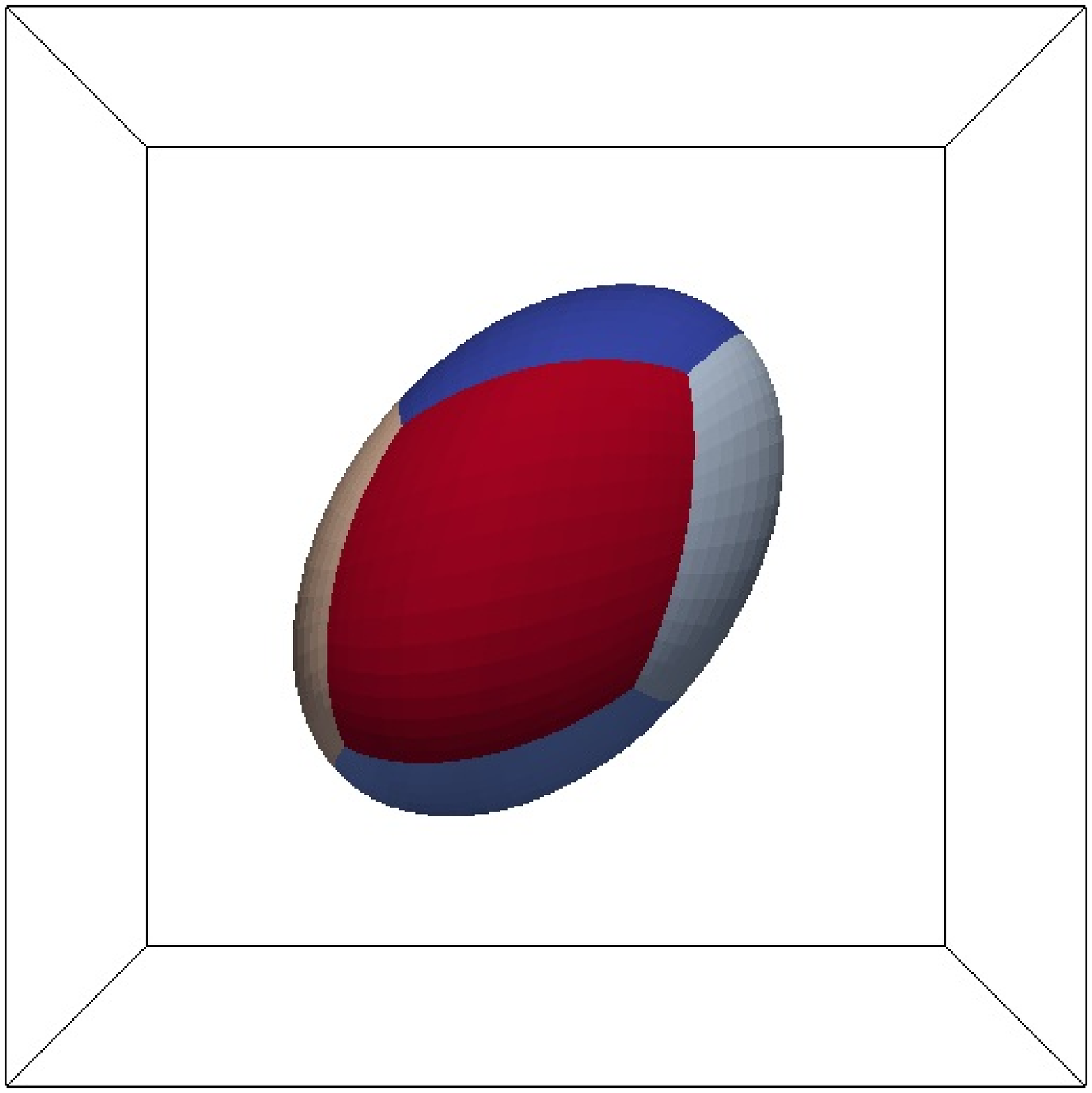}
\end{subfigure}
\begin{subfigure}{0.32\textwidth}
\includegraphics[scale=0.19]{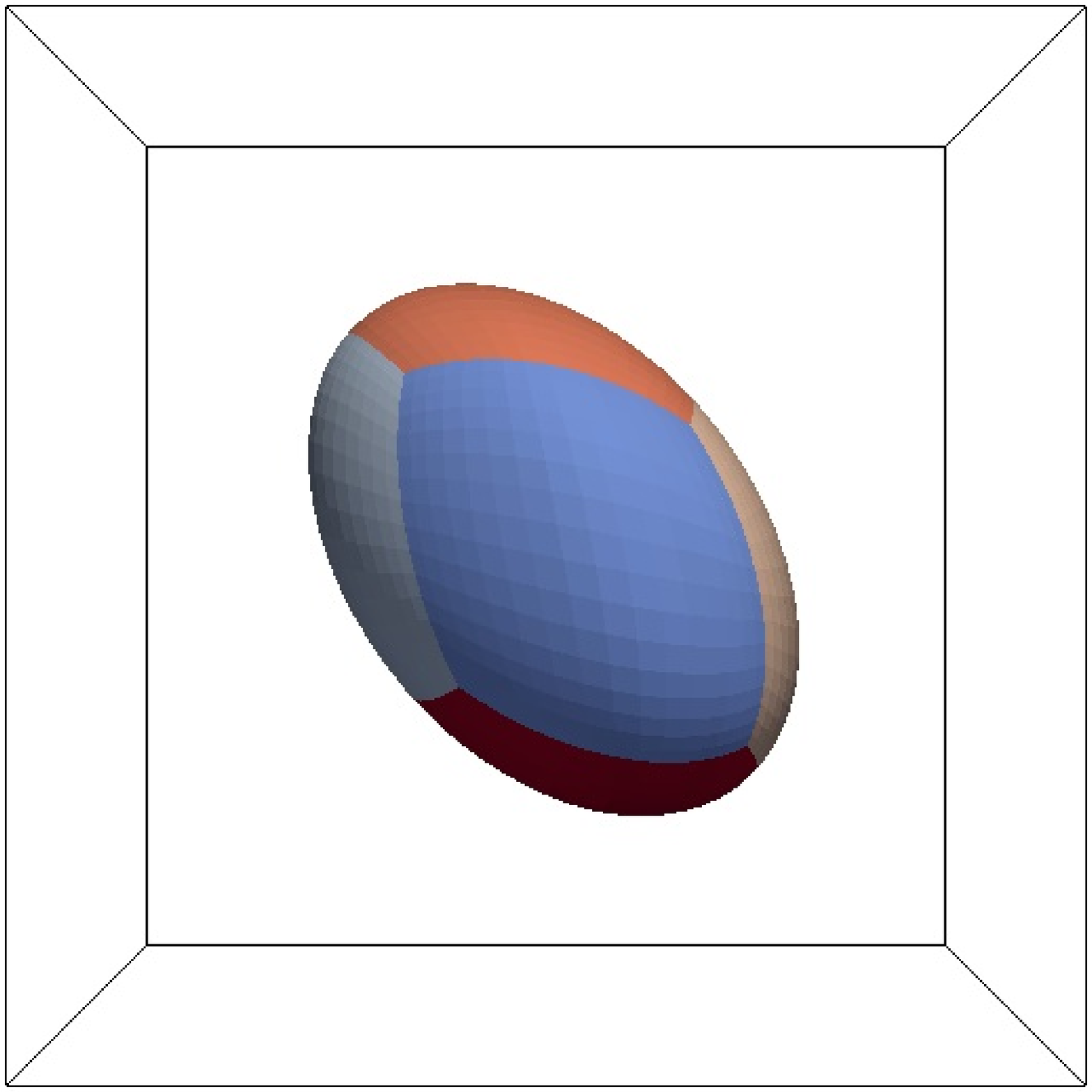}
\end{subfigure}
\caption{Resulting shape in case of the initial guess is 
$B_{0.3}({\bs 0})$ and the desired tensor $\boldsymbol{B}_4$.}
\label{result_rotated}
\end{figure}
%
%=============================================
\subsection{Cube}
%=============================================
In the next examples, we consider the cube 
$\mathcal{C} = [-0.15, 0.15]^3$ and its rotated version
$\boldsymbol{T}\mathcal{C}$, respectively, as initial guesses. We want 
to examine what happens if we optimize its shape with respect to a 
given desired effective tensor. In all examples, convergence 
to the final value was achieved between 13 and 16 iterations.

If the effective tensor is diagonal relative to the axes of the 
initial cube, then we observe that $\Omega$ takes the shape 
of a cuboid, aligned with these axes. If not, then $\Omega$ takes 
approximately the shape of a parallelepiped. This behaviour can 
clearly be observed in Figures~\ref{fig:B_3|C}--\ref{fig:B_4|TC}.
Therein, the first situation (initial guess and desired effective tensor 
aligned) appears in Figures~\ref{fig:B_3|C} and \ref{fig:B_4|TC}, 
while the second situation (initial guess and desired effective 
tensor not aligned) appears in Figures~\ref{fig:B_4|C} and 
\ref{fig:B_3|TC}.

\begin{figure}[hbt]
\begin{subfigure}{0.32\textwidth}
\includegraphics[scale=0.19]{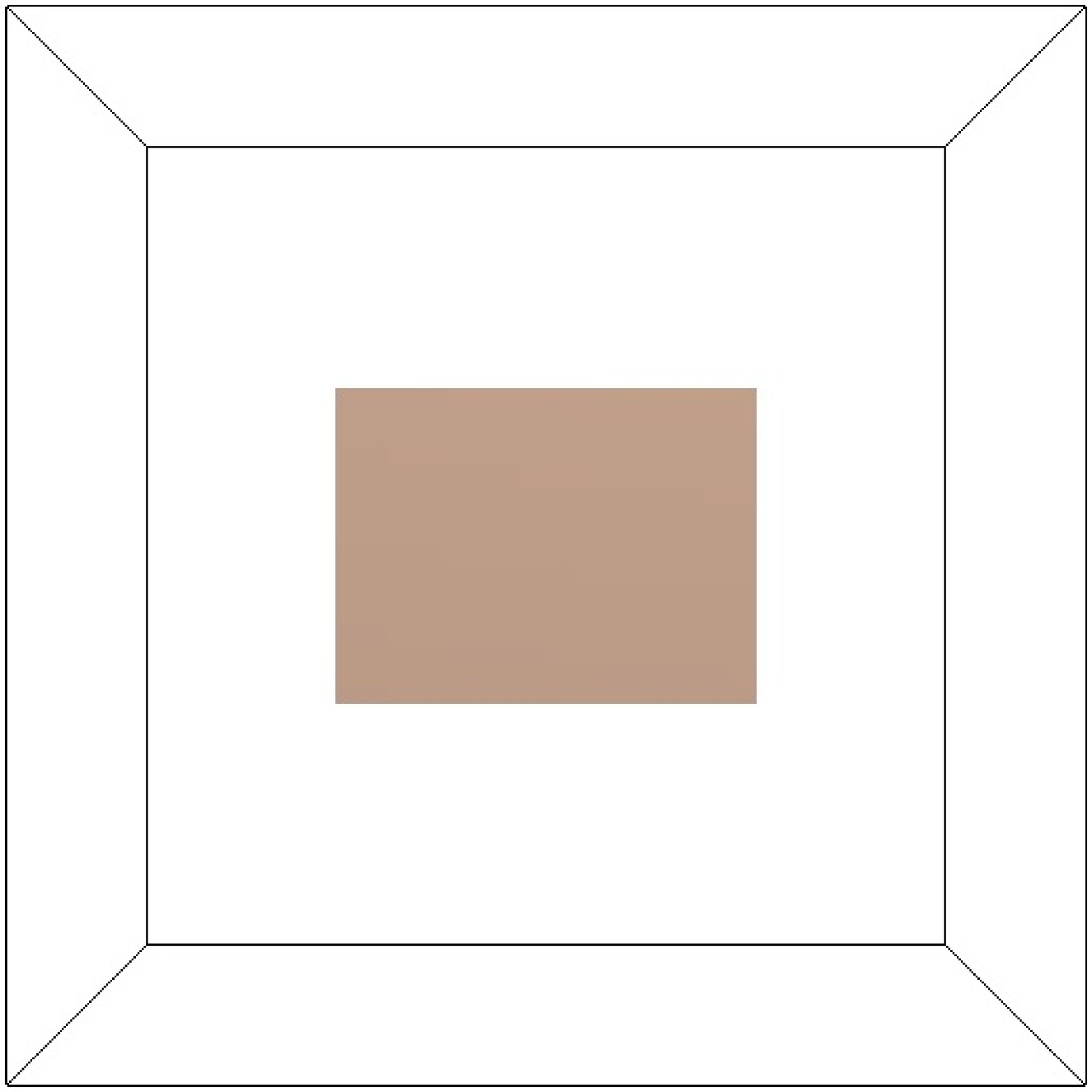}
\end{subfigure}
\begin{subfigure}{0.32\textwidth}
\includegraphics[scale=0.19]{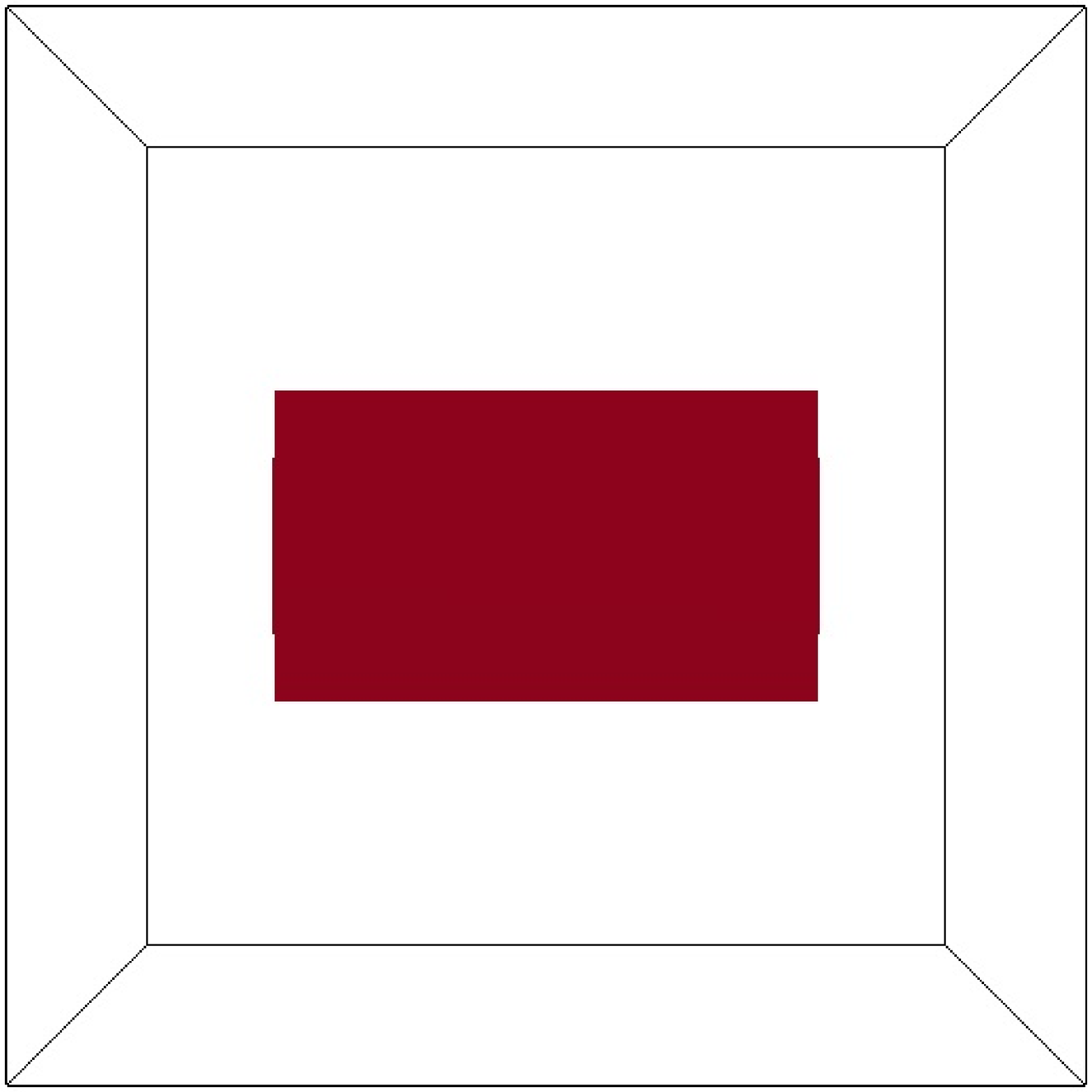}
\end{subfigure}
\begin{subfigure}{0.32\textwidth}
\includegraphics[scale=0.19]{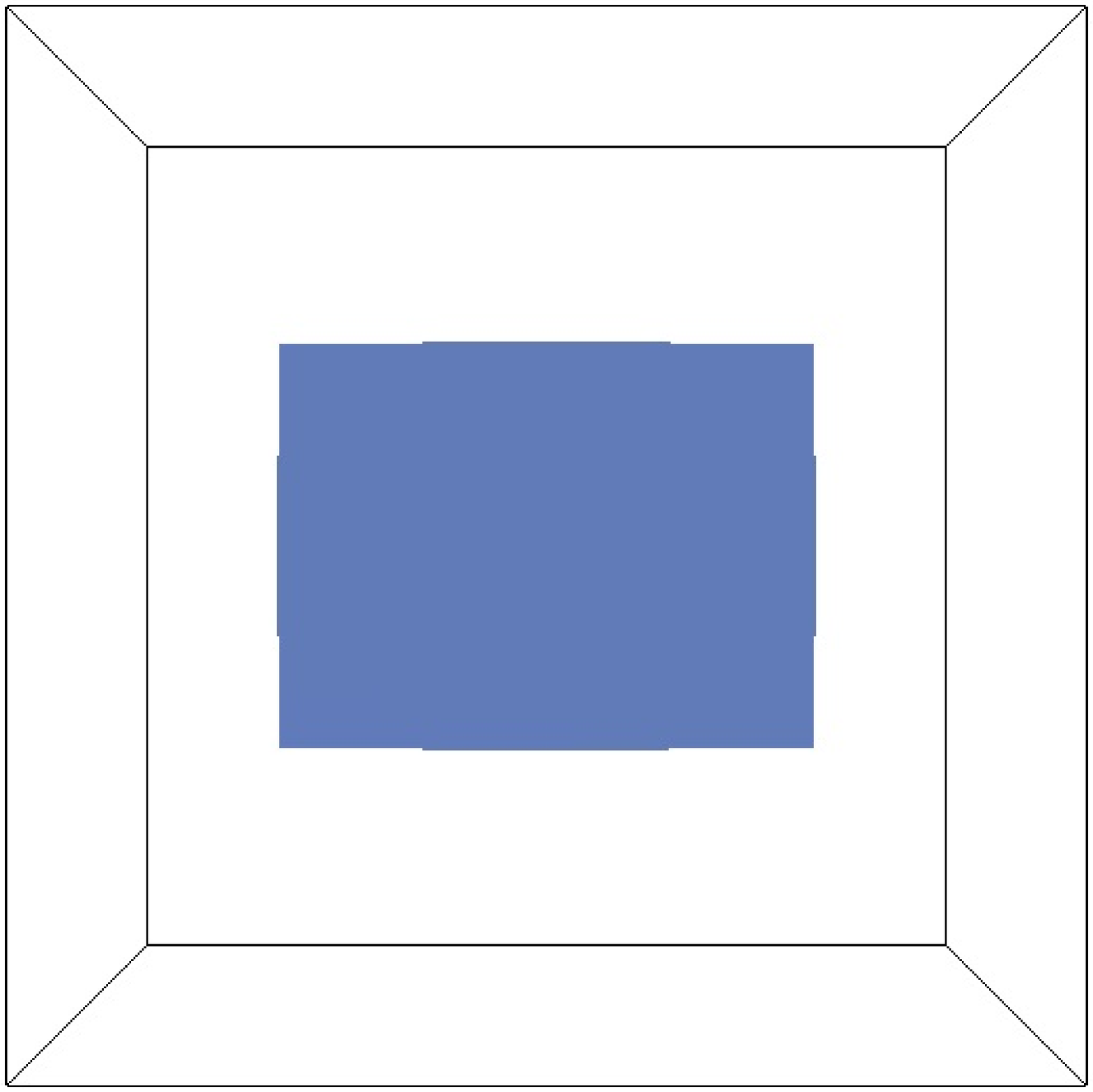}
\end{subfigure}
\caption{Resulting shape in case of the initial guess is 
$\mathcal{C}$ and the desired tensor $\boldsymbol{B}_3$.}
\label{fig:B_3|C}
\end{figure}

\begin{figure}[hbt]
\begin{subfigure}{0.32\textwidth}
\includegraphics[scale=0.19]{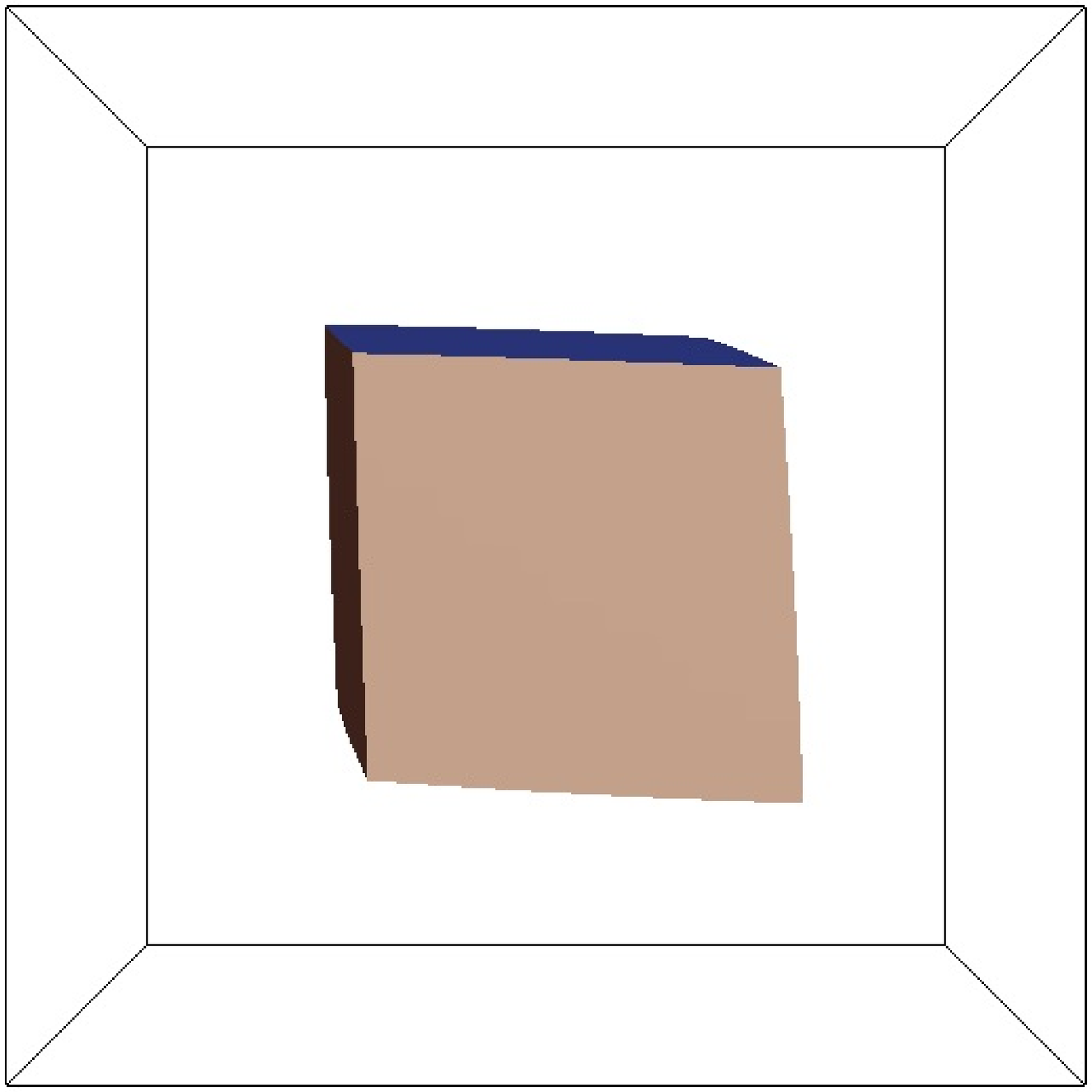}
\end{subfigure}
\begin{subfigure}{0.32\textwidth}
\includegraphics[scale=0.19]{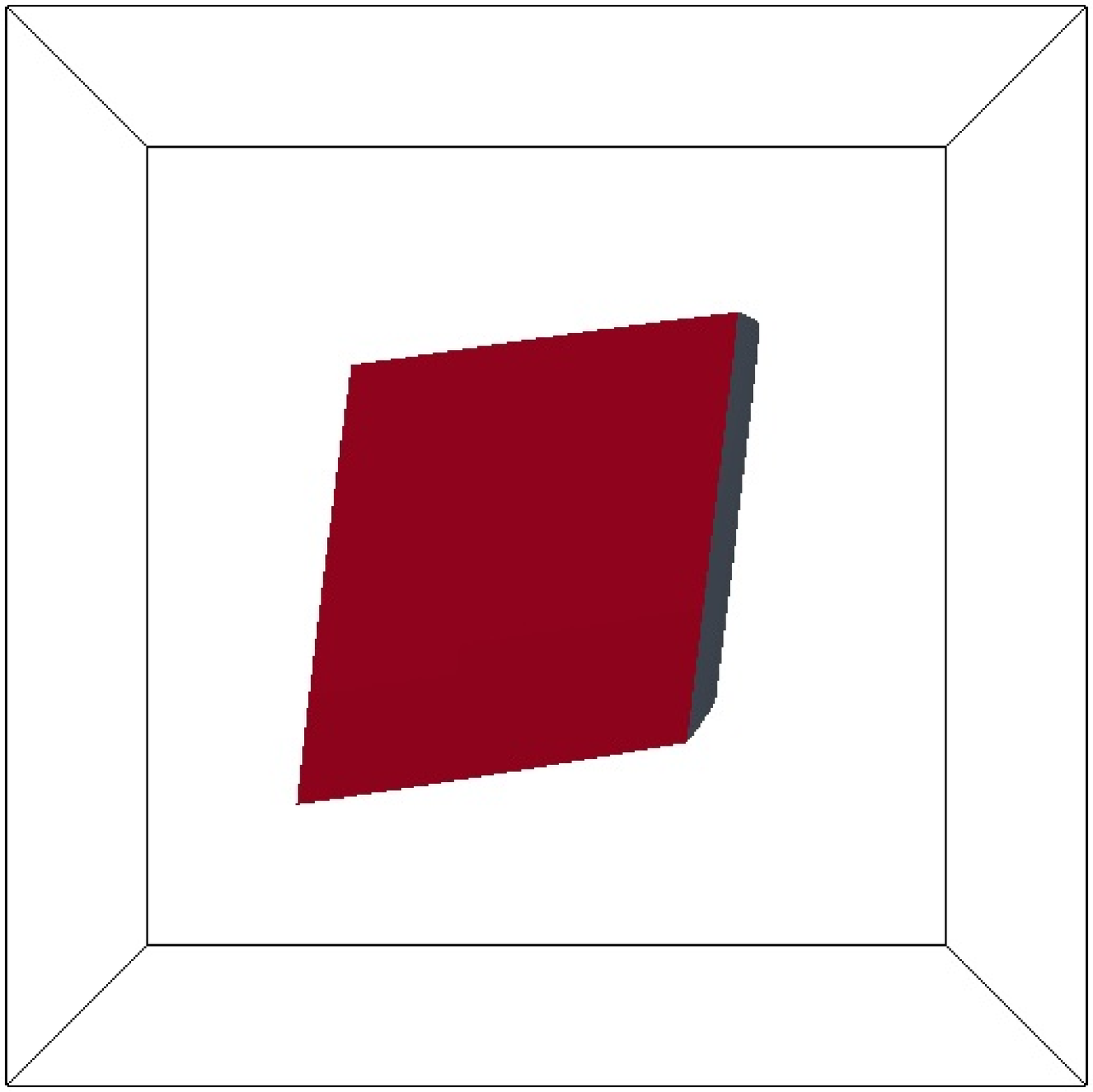}
\end{subfigure}
\begin{subfigure}{0.32\textwidth}
\includegraphics[scale=0.19]{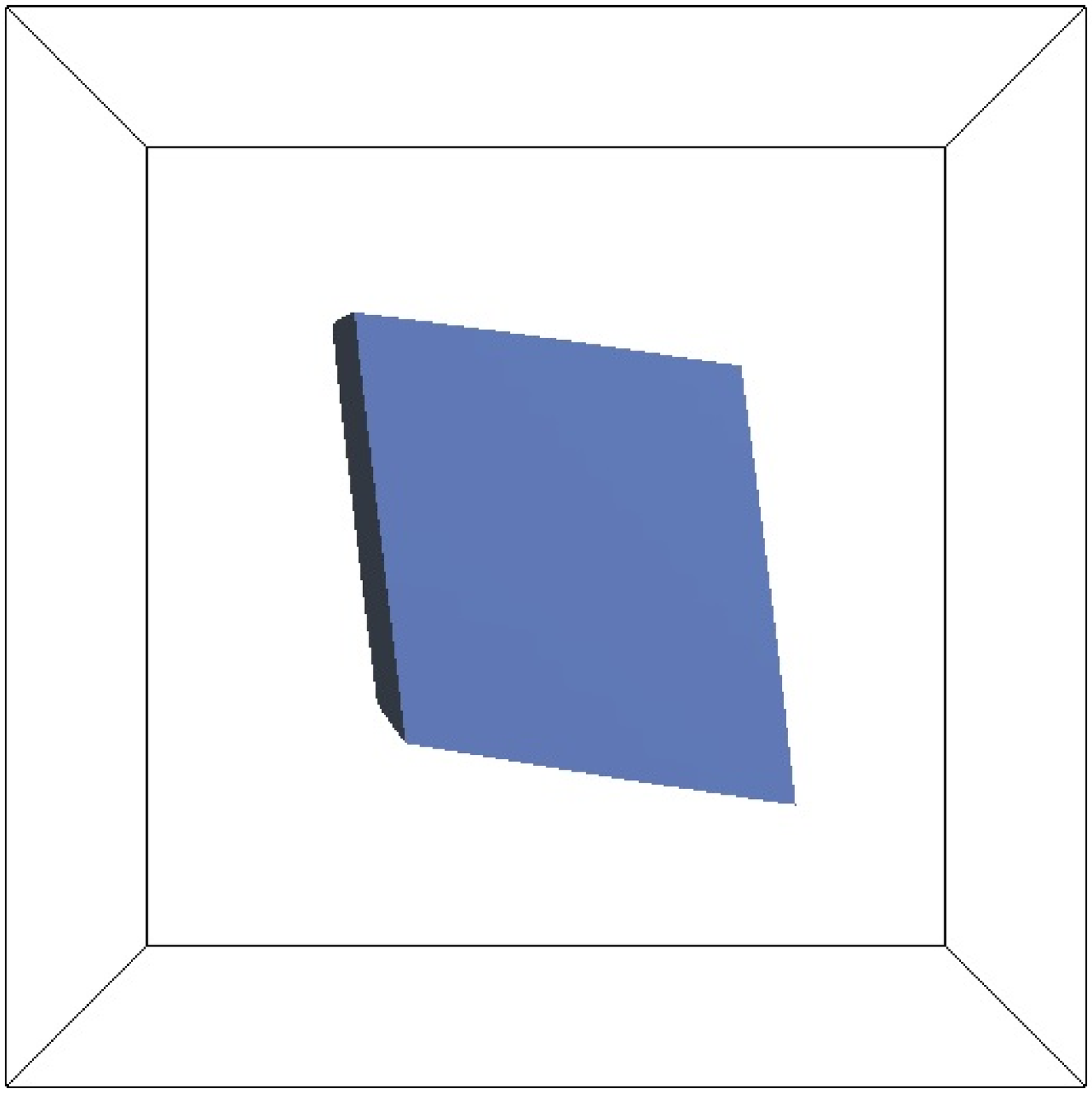}
\end{subfigure}
\caption{Resulting shape in case of the initial guess is 
$\mathcal{C}$ and the desired tensor $\boldsymbol{B}_4$.}
\label{fig:B_4|C}
\end{figure}

\begin{figure}[hbt]
\begin{subfigure}{0.32\textwidth}
\includegraphics[scale=0.19]{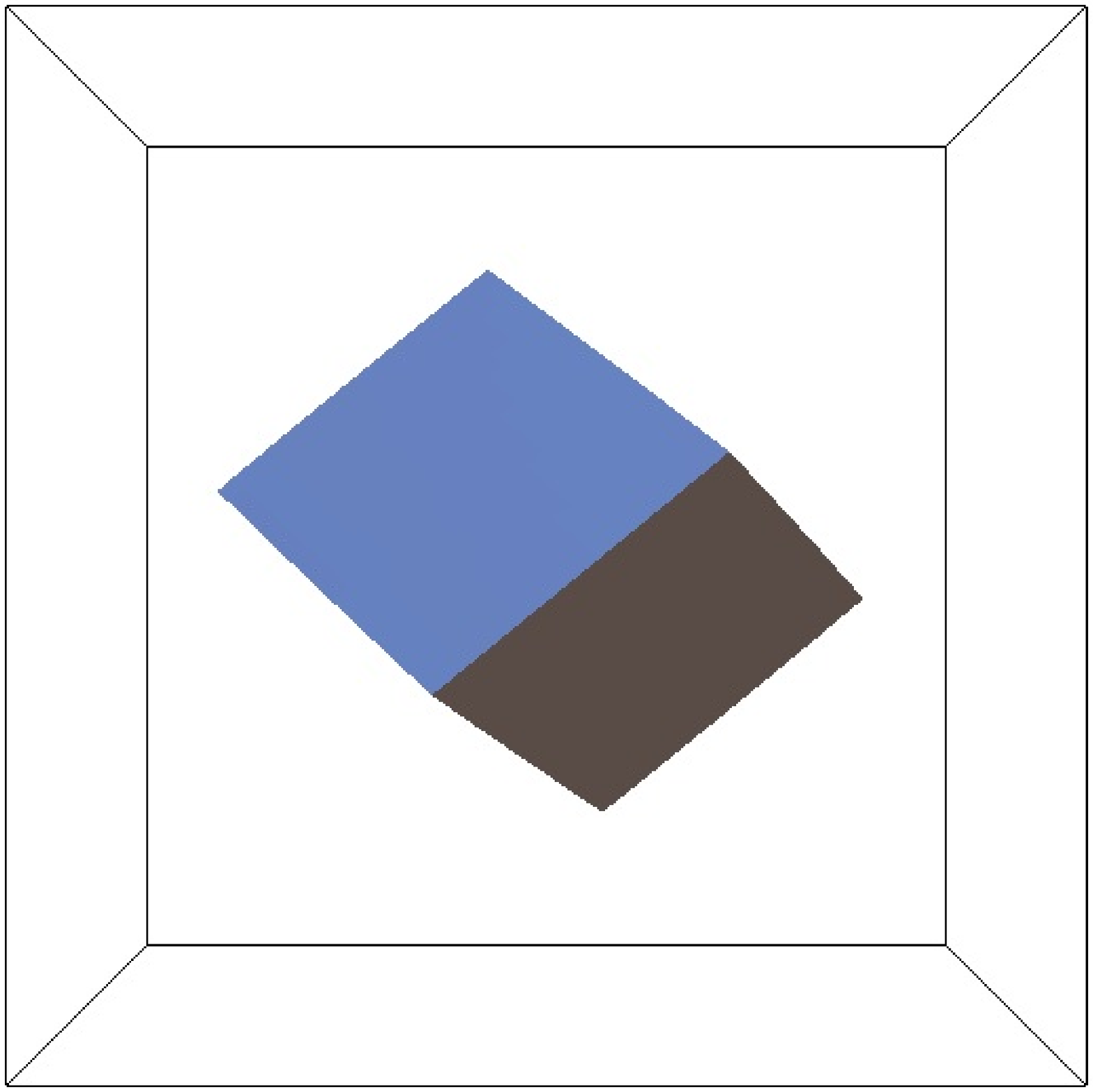}
\end{subfigure}
\begin{subfigure}{0.32\textwidth}
\includegraphics[scale=0.19]{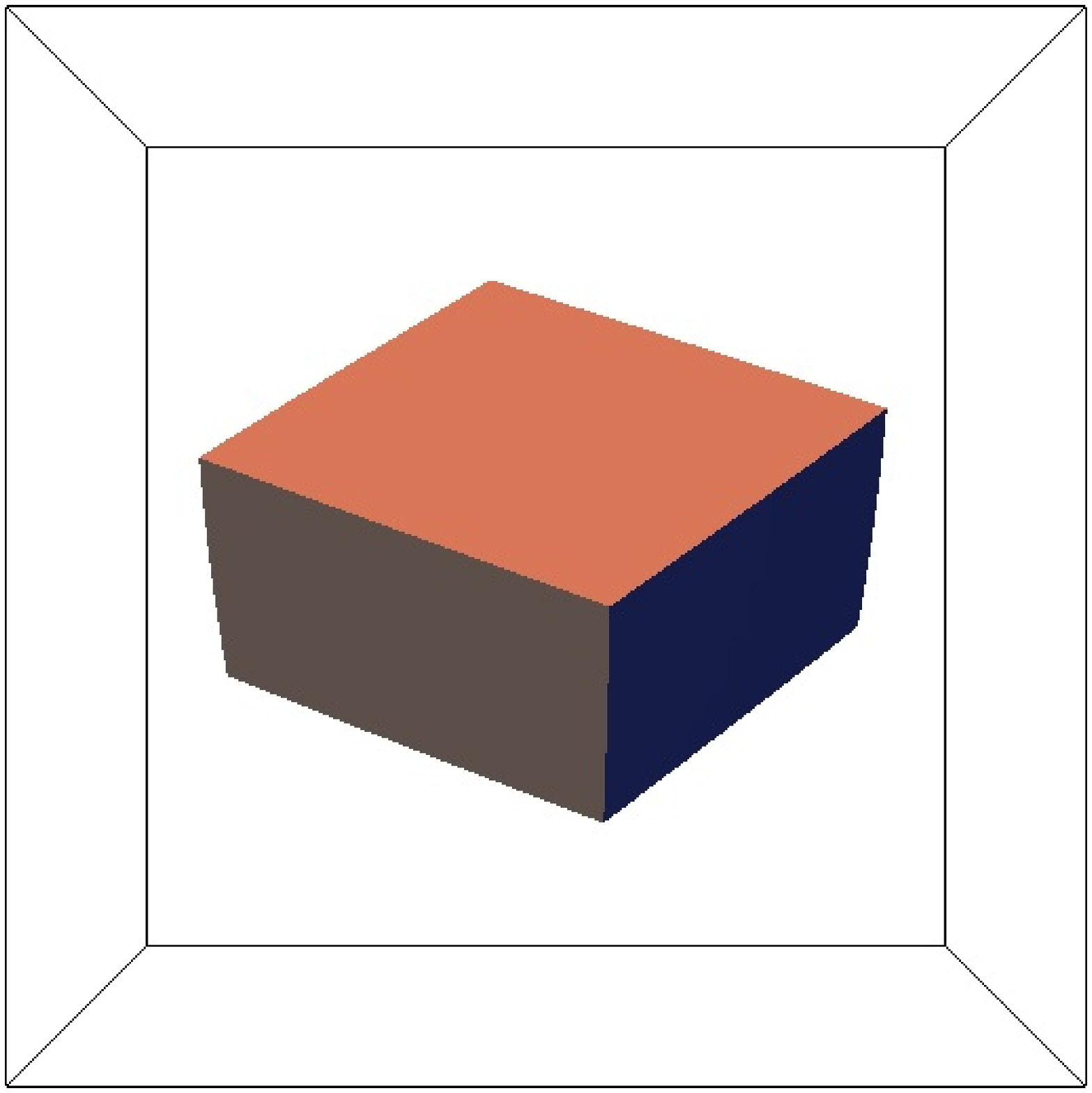}
\end{subfigure}
\begin{subfigure}{0.32\textwidth}
\includegraphics[scale=0.19]{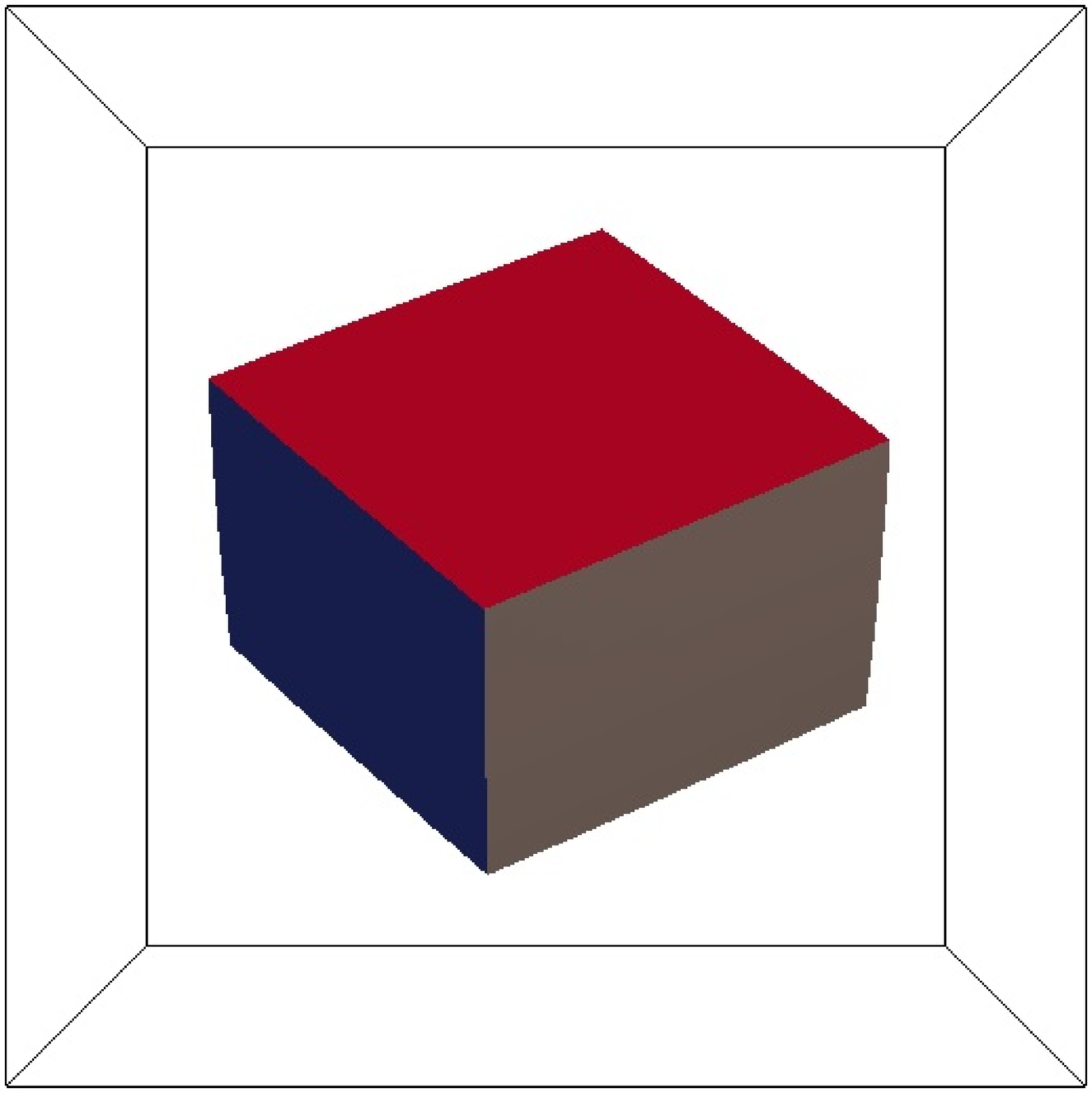}
\end{subfigure}
\caption{Resulting shape in case of the initial guess is 
$\boldsymbol{T}\mathcal{C}$ and the desired tensor $\boldsymbol{B}_3$.}
\label{fig:B_3|TC}
\end{figure}

\begin{figure}[hbt]
\begin{subfigure}{0.32\textwidth}
\includegraphics[scale=0.19]{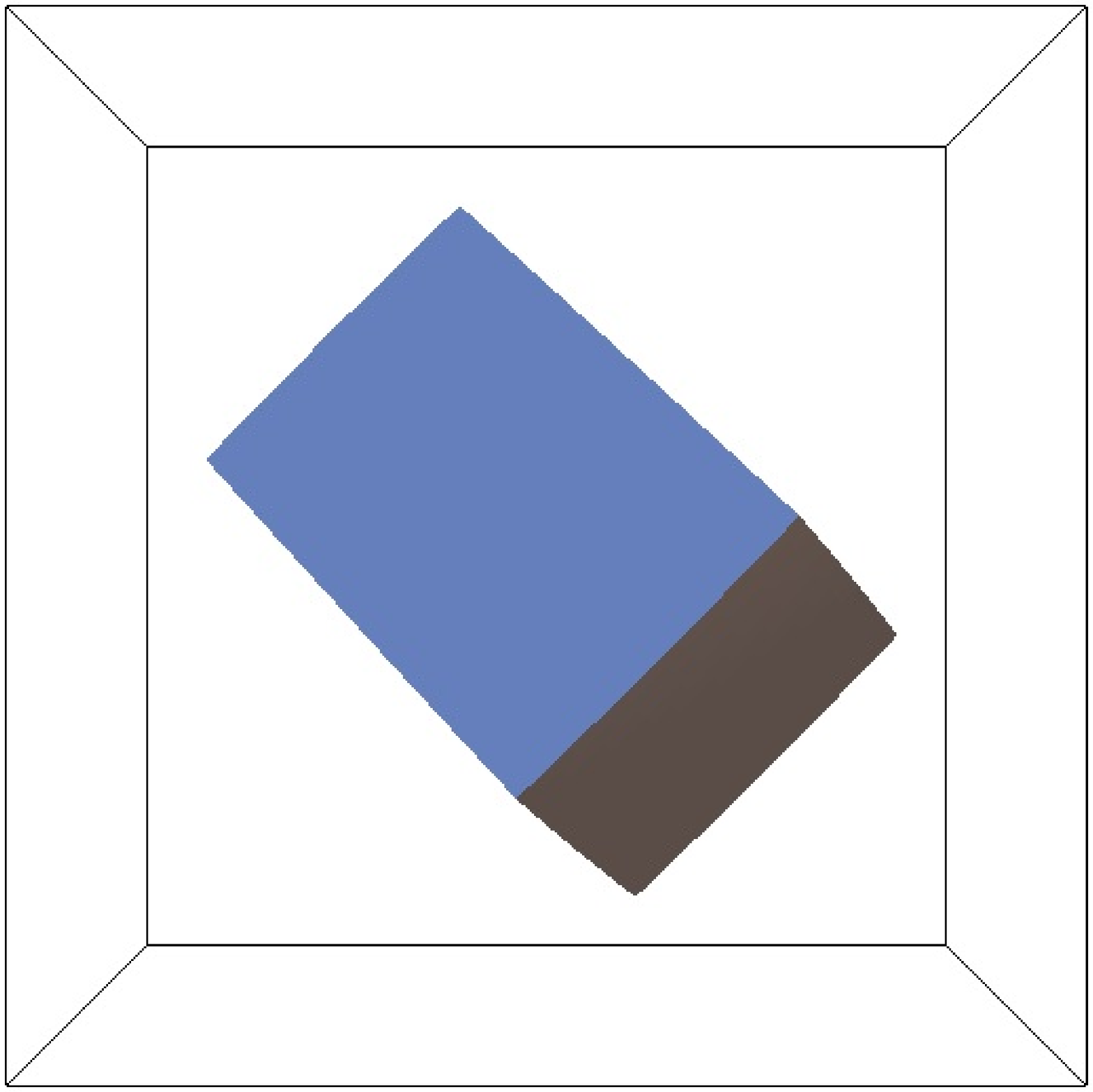}
\end{subfigure}
\begin{subfigure}{0.32\textwidth}
\includegraphics[scale=0.19]{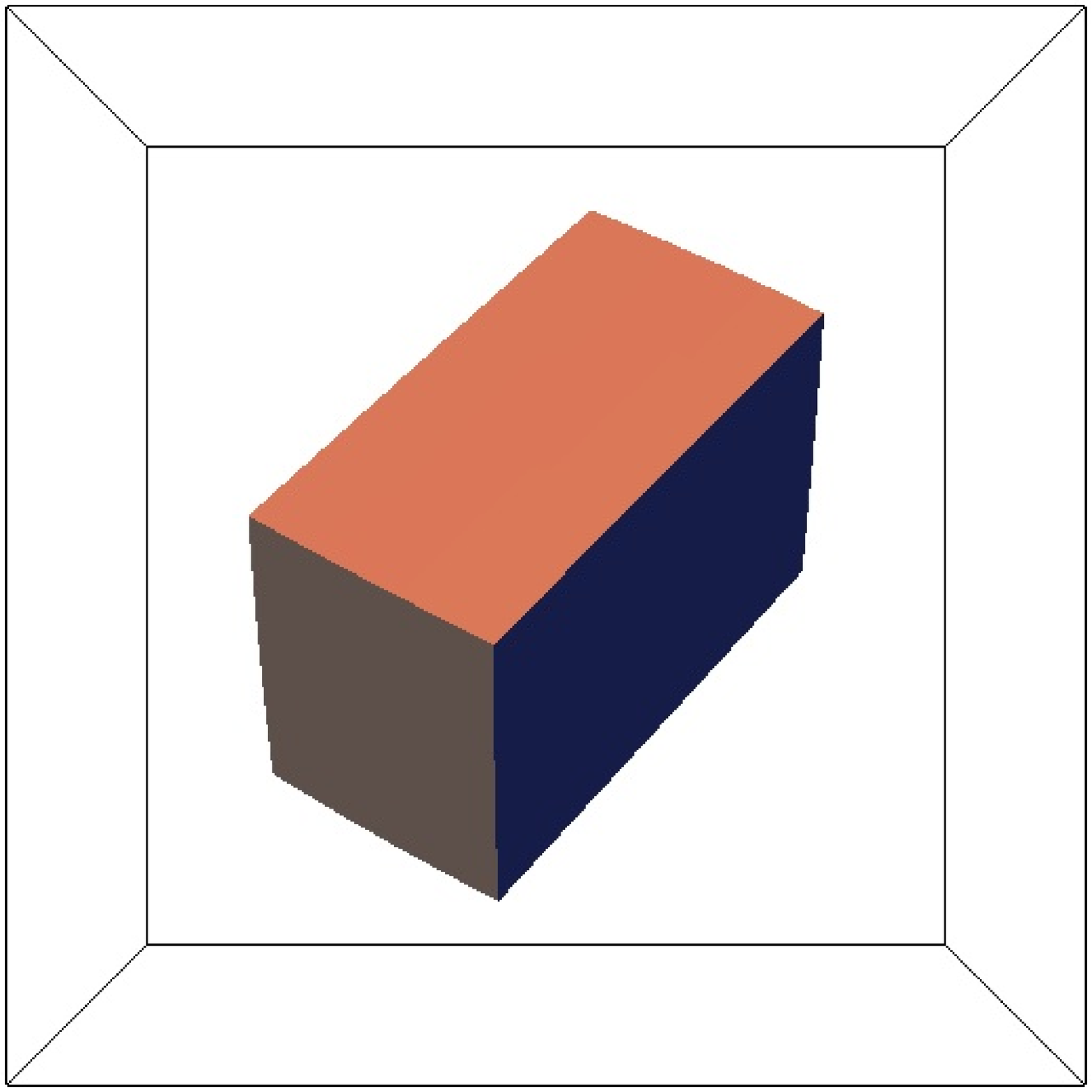}
\end{subfigure}
\begin{subfigure}{0.32\textwidth}
\includegraphics[scale=0.19]{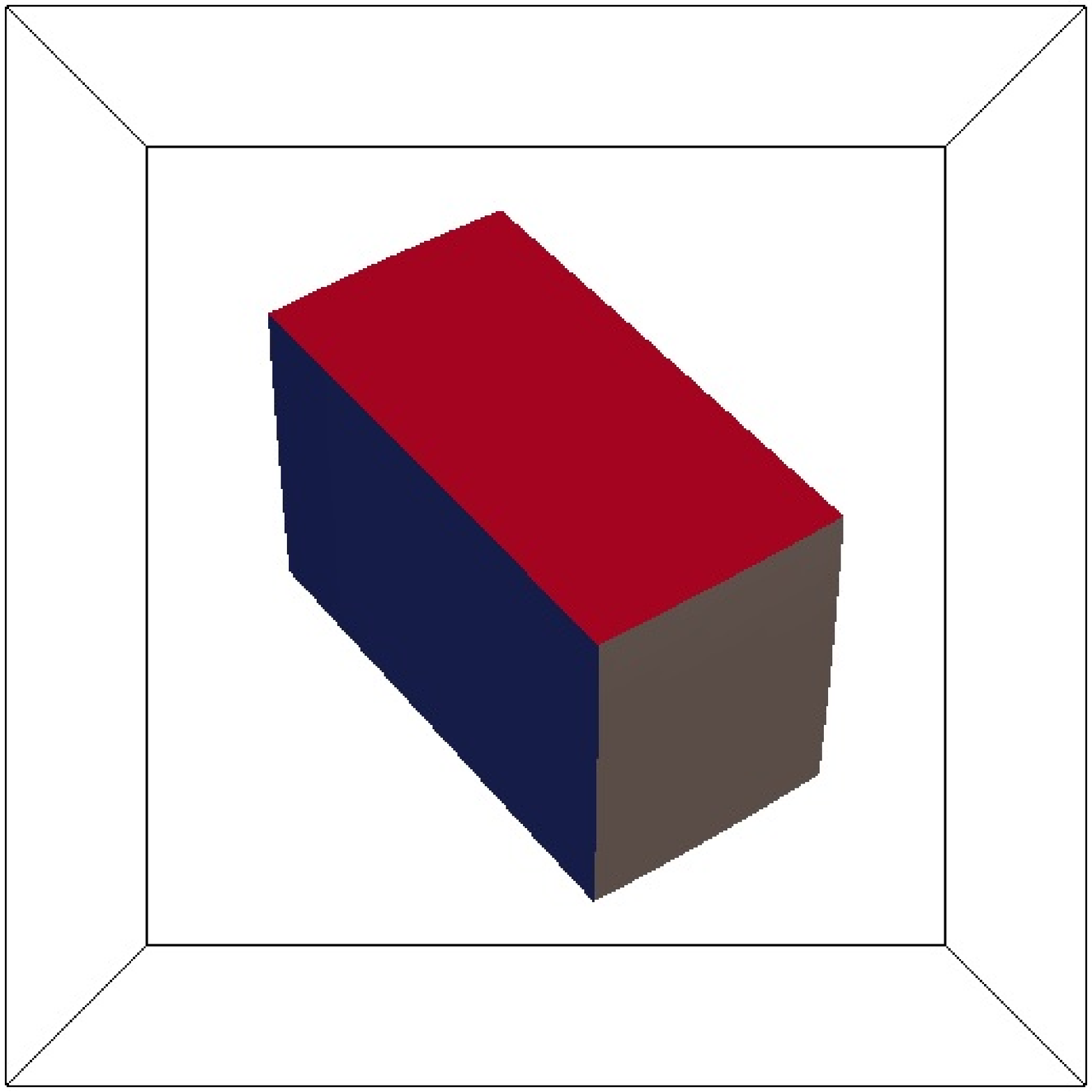}
\end{subfigure}
\caption{Resulting shape in case of the initial guess is 
$\boldsymbol{T}\mathcal{C}$ and the desired tensor $\boldsymbol{B}_4$.}
\label{fig:B_4|TC}
\end{figure}

%============================================
\subsection{Sphere and cube}
\label{section_sphere_cube}
%============================================
Another important configuration we investigate is the situation of
two 
initial shapes which are not connected but placed at different 
locations in the micro cell $Y$. To study this case, we define
the geometry
\[
\mathcal{D} = B_{0.15}\big(-[0.25,0.25,0.25]^\intercal\big)\cup[0.175,0.325]^3.
\]
The desired material tensor under consideration is the diagonal
matrix
\[
\boldsymbol{B}_5 = \begin{bmatrix}1 & & \\ & 0.995 & \\ &  &0.99\end{bmatrix}.
\]
As can be seen from Figure \ref{figure_normal_ell}, using $\mathcal{D}$
as initial guess, the geometries do not just deform as before, but rather 
affect each other. Especially the cube is stretched towards the (deformed) 
sphere.

\begin{figure}[hbt]
\begin{subfigure}{0.32\textwidth}
\includegraphics[scale=0.19]{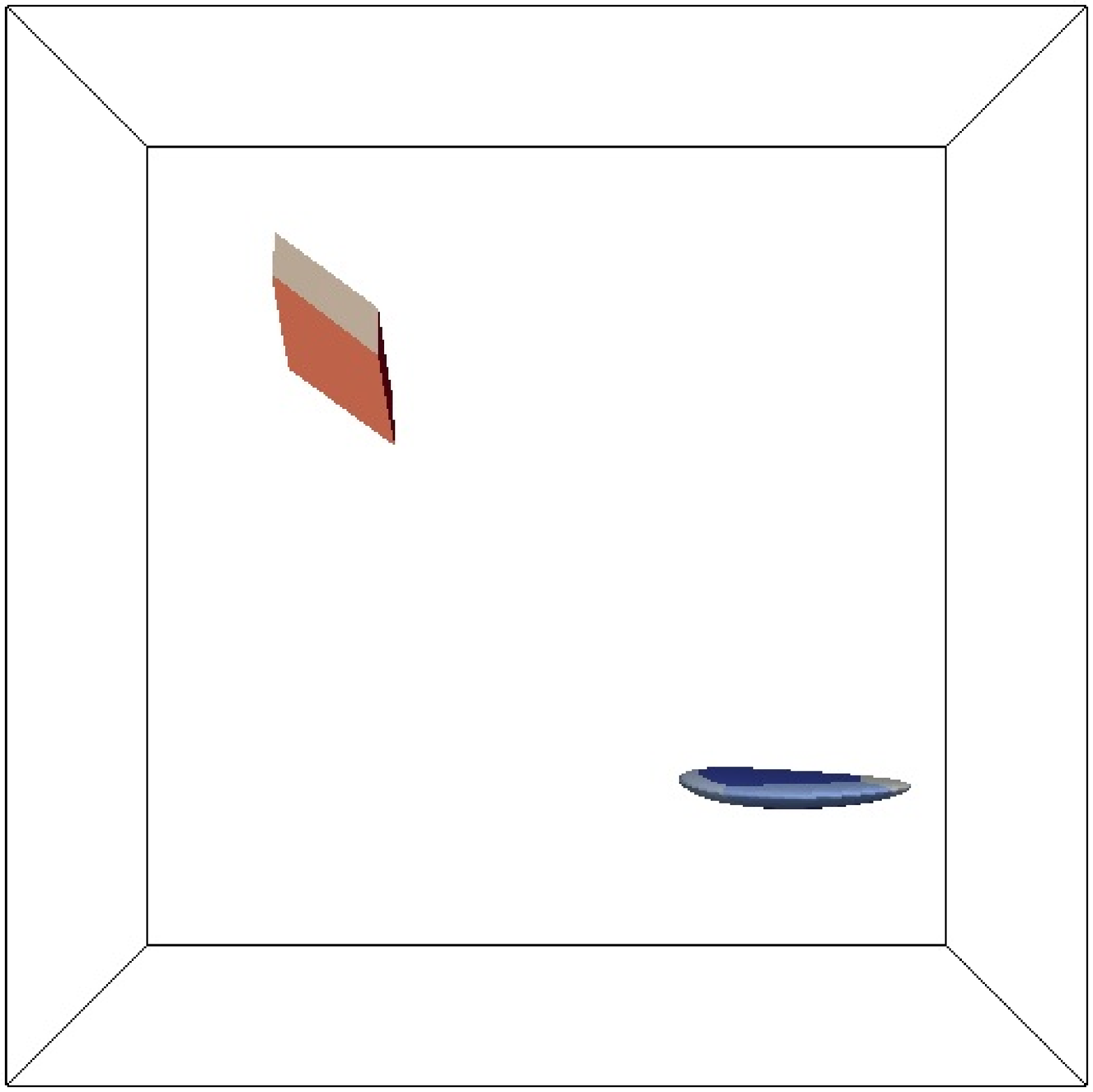}
\end{subfigure}
\begin{subfigure}{0.32\textwidth}
\includegraphics[scale=0.19]{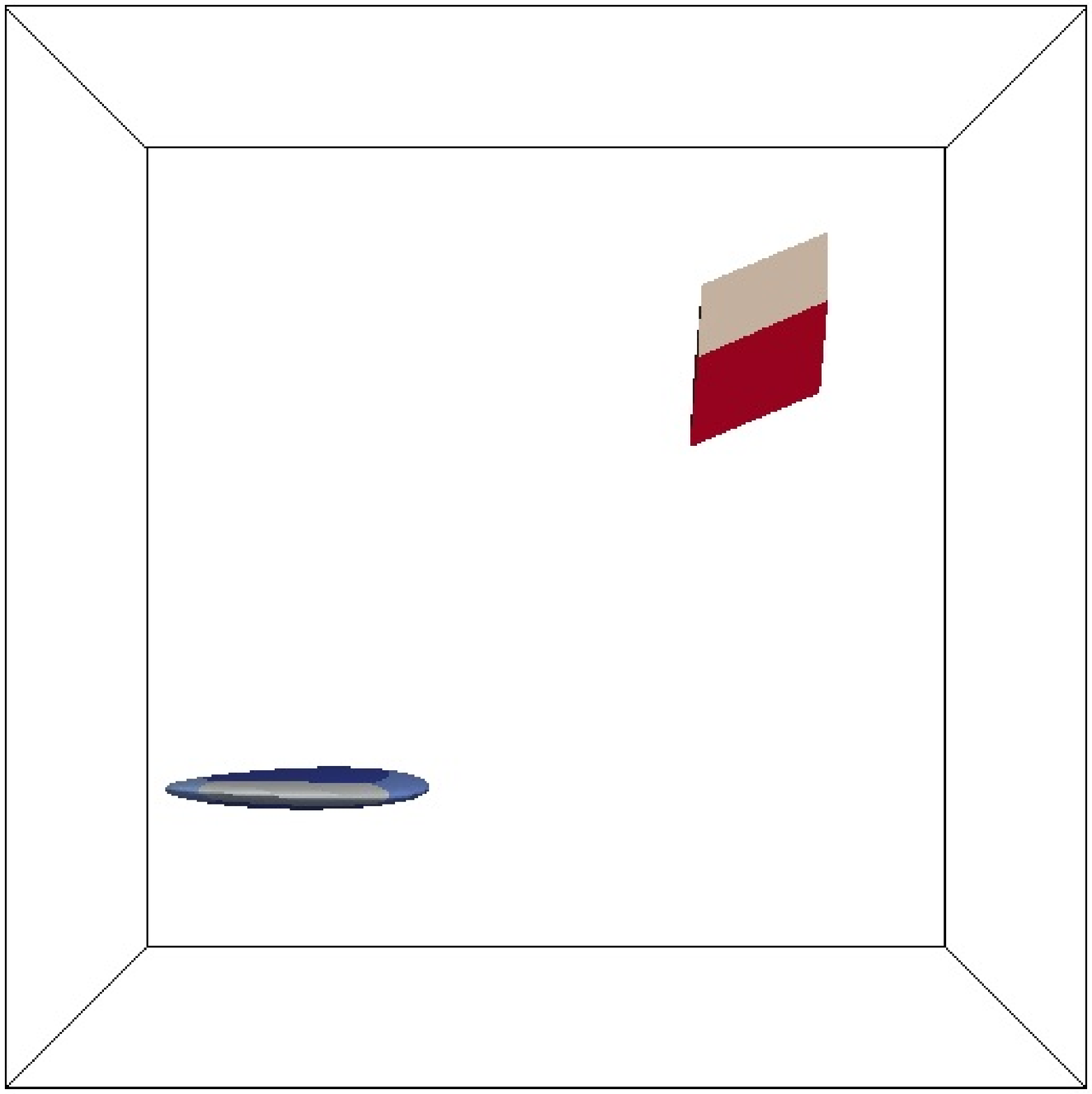}
\end{subfigure}
\begin{subfigure}{0.32\textwidth}
\includegraphics[scale=0.19]{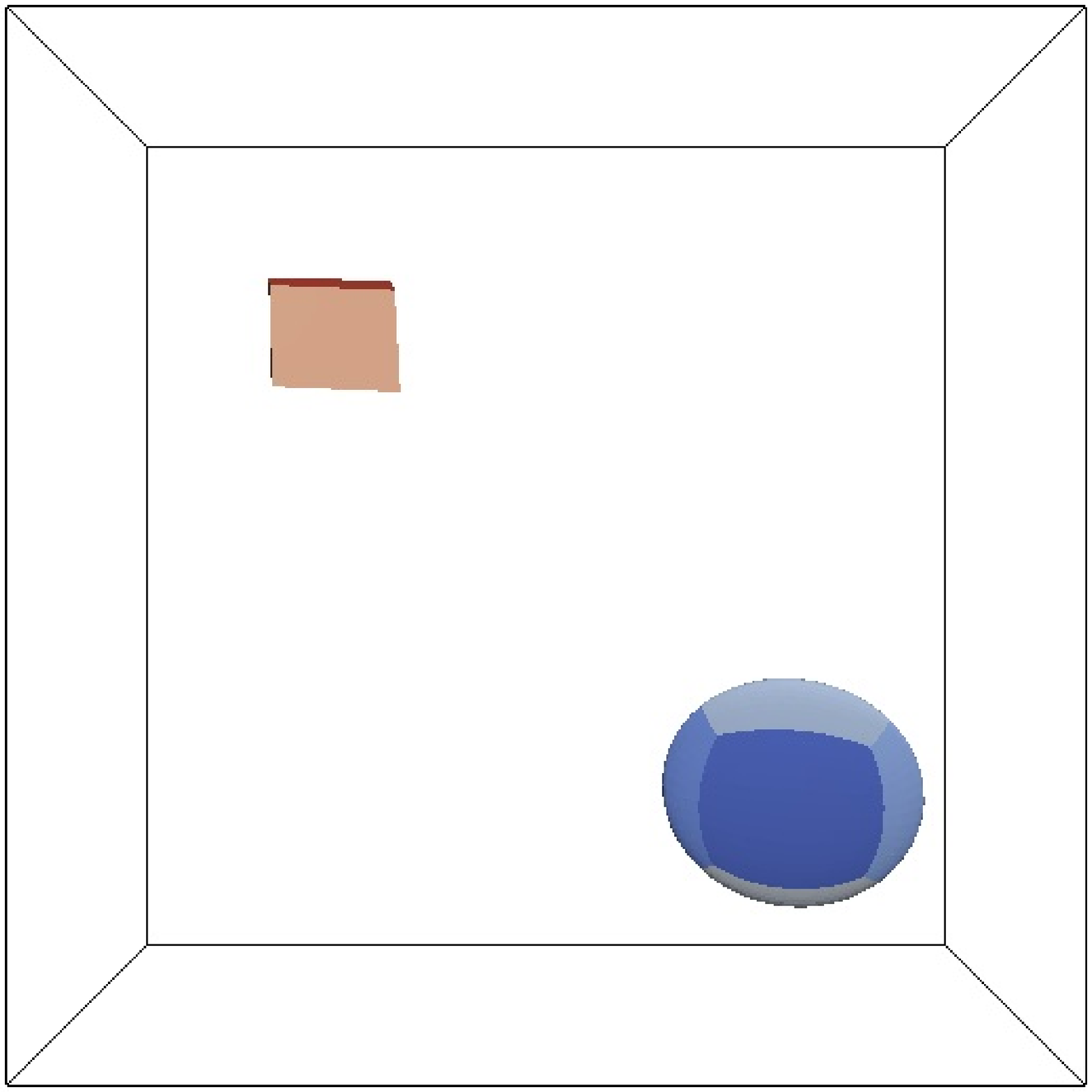}
\end{subfigure}
\caption{Desired tensor $\boldsymbol{B}_5$, initial guess $\mathcal{D}$, 
$\ell = 1$. Convergence is achieved after 12 iterations, computed on the fourth refinement level.}
\label{figure_normal_ell}
\end{figure}

Instead of considering different desired tensors as in the examples 
before, we follow a different approach here. Recall that the 
construction of the displacement fields depends on the 
Mat\'ern kernel \eqref{eq:matern}, where $\ell>0$ denotes 
the correlation length. If $\ell$ is chosen small, then the 
two shapes act more like two different geometries, compare 
Figure \ref{figure_small_ell}. However, $p=16$ are too 
few displacement fields to decrease the value of the 
objective functional within a reasonable amount of iterations.
Hence, we rather work with $p=50$ displacement fields here, 
resulting in a rapid convergence. Note that the number
of iterations to reduce the value of the shape functional
below $10^{-5}$ has been 12 when $\ell=1$ and 8 when 
$\ell=1/4$. 

\begin{figure}[hbt]
\begin{subfigure}{0.32\textwidth}
\includegraphics[scale=0.19]{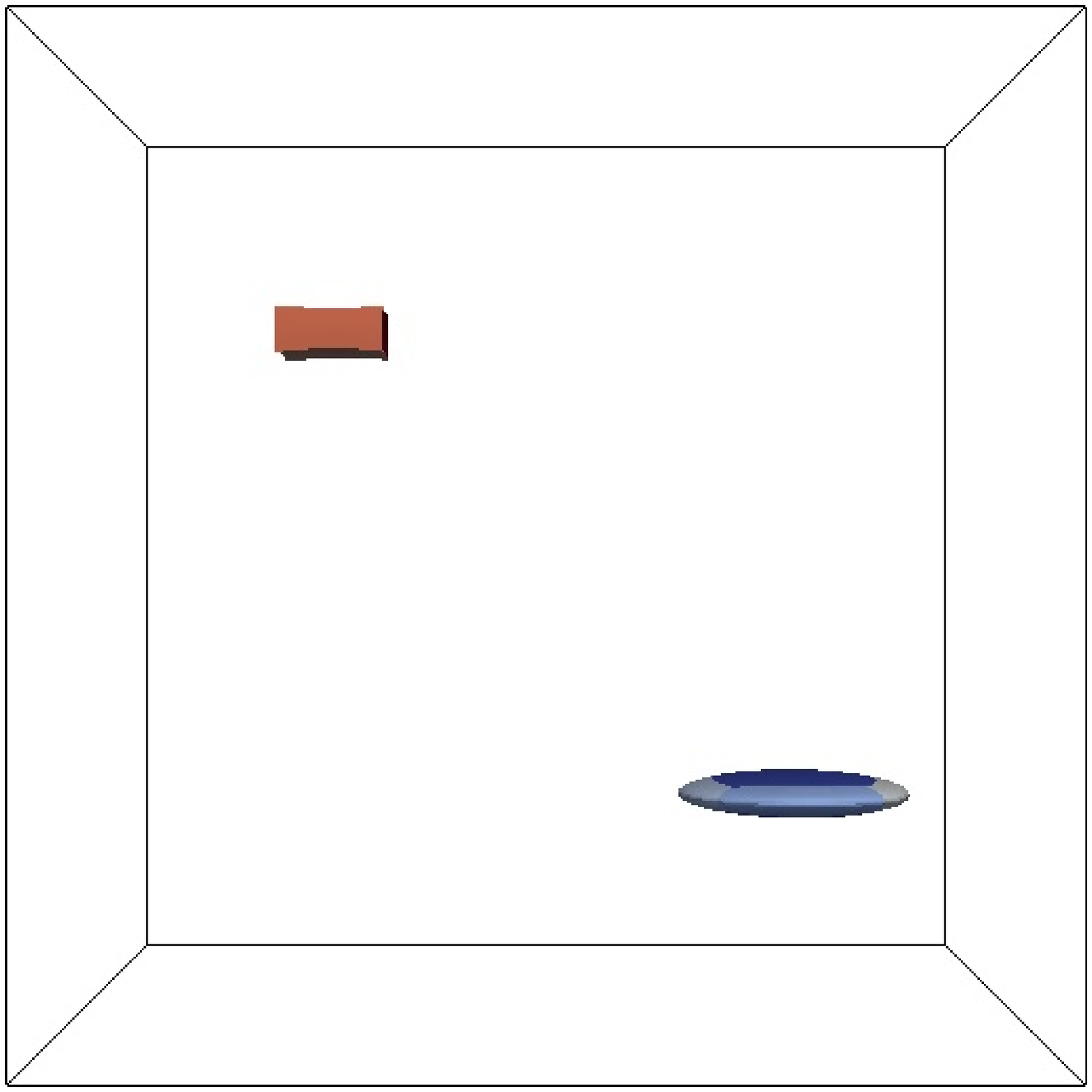}
\end{subfigure}
\begin{subfigure}{0.32\textwidth}
\includegraphics[scale=0.19]{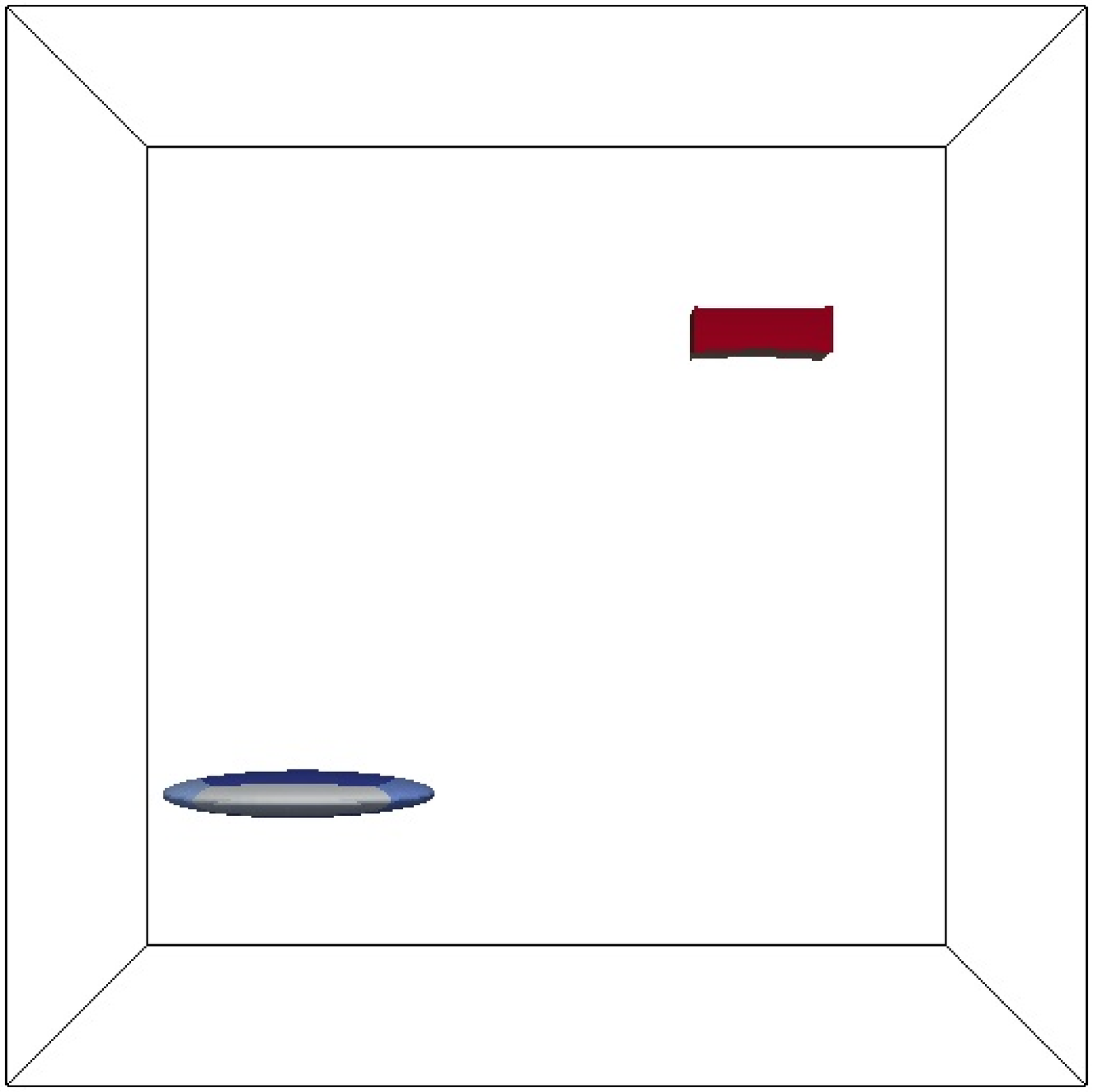}
\end{subfigure}
\begin{subfigure}{0.32\textwidth}
\includegraphics[scale=0.19]{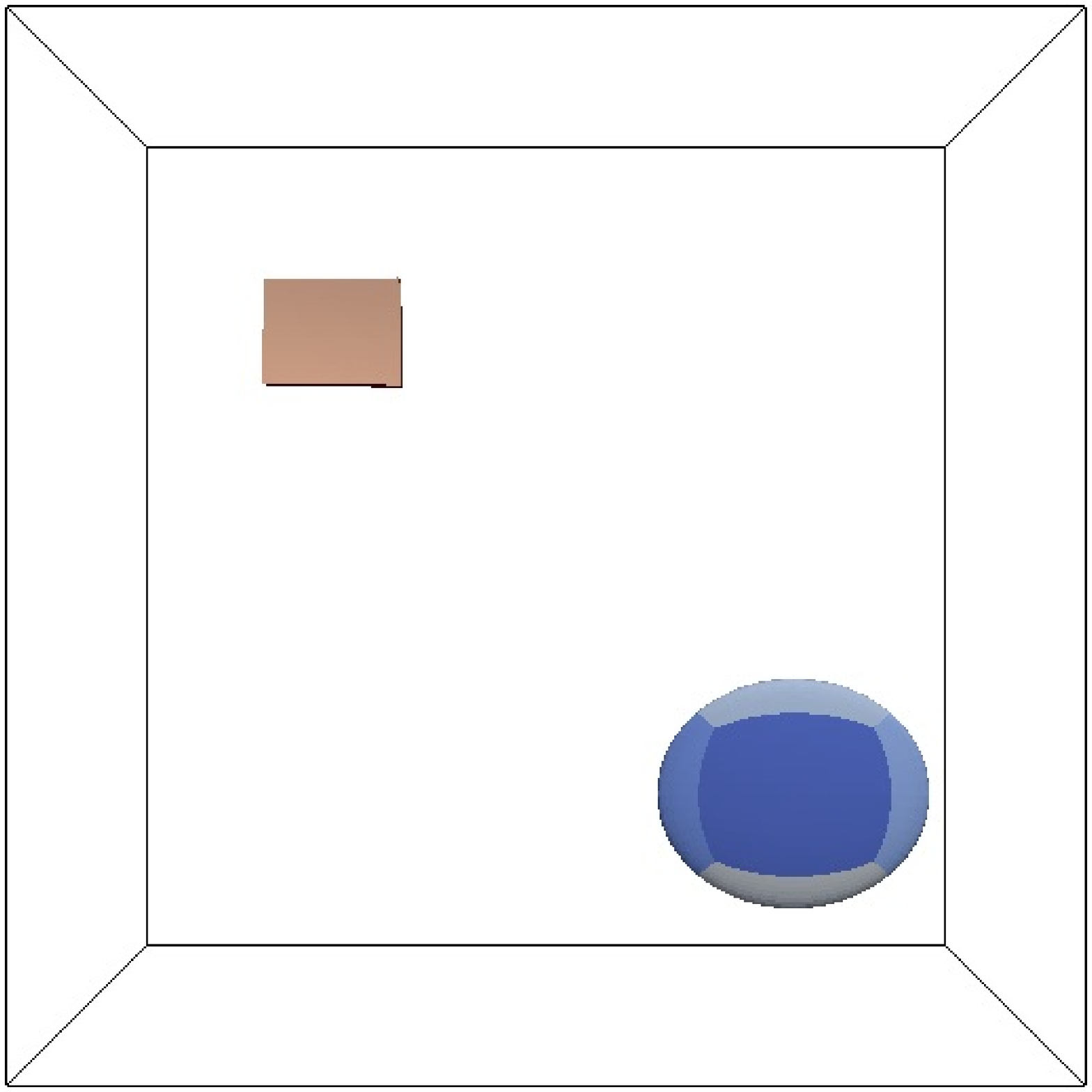}
\end{subfigure}
\caption{Desired tensor $\boldsymbol{B}_3$, initial guess $\mathcal{D}$, 
$\ell = 1/4$. Convergence is achieved 
after 8 iterations, computed on the fourth refinement level.}
\label{figure_small_ell}
\end{figure}

%============================================
\subsection{Drilled cube}\label{sct:drilled}
\label{section_drilled_cube}
%============================================
In our final example, we consider a more complex geometry
as initial guess, namely a cube $[-0.3,0.3]^3$, where
three holes of diameter 0.15 have been drilled into, 
compare Figure~\ref{fig:toy}. Notice that the parametrization 
of the cube is represented by 48 patches. 

\begin{figure}[hbt]
\includegraphics[width=4.5cm,clip,trim=1200 50 900 350]{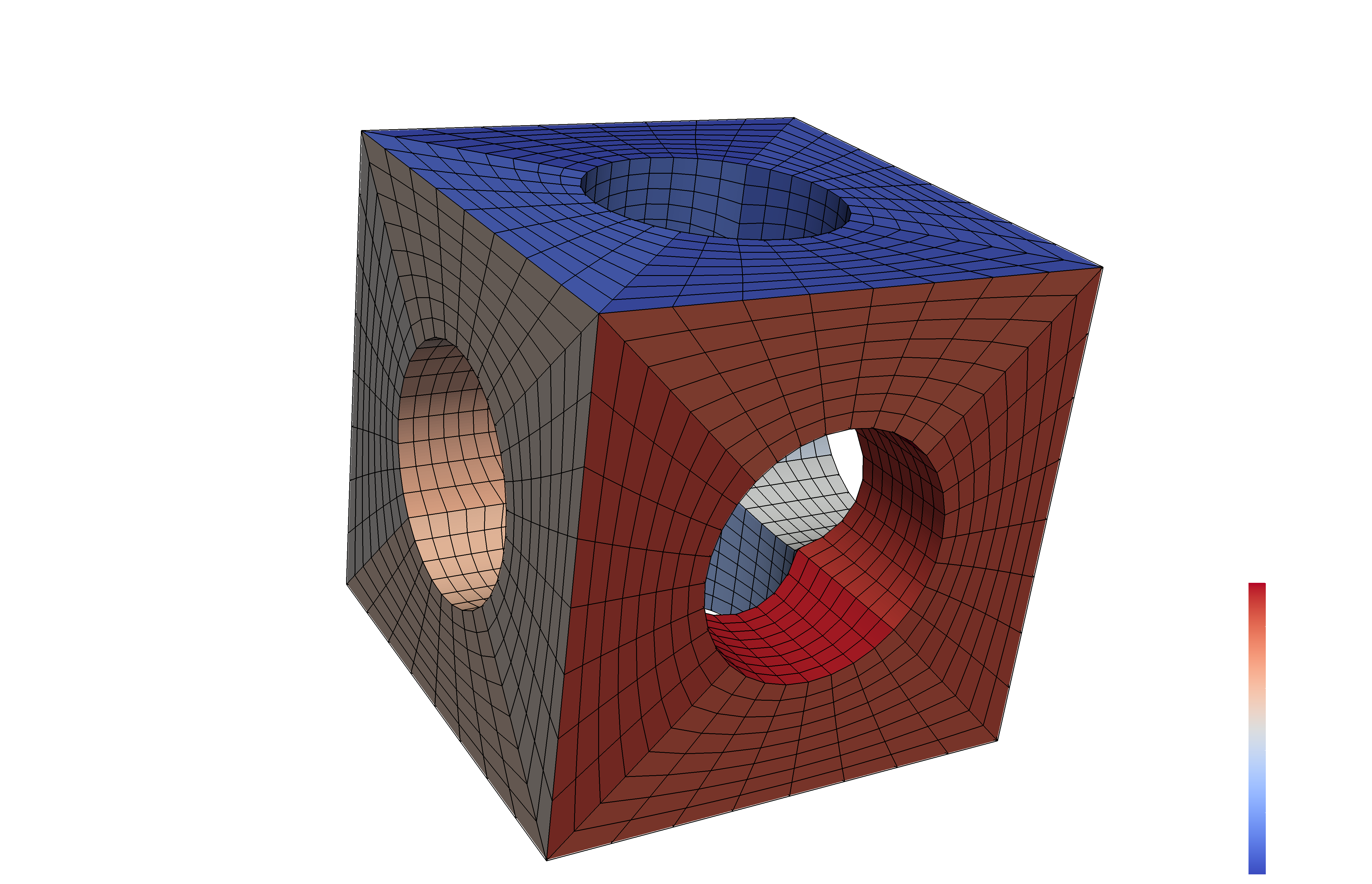}
\caption{The parametrization of the drilled cube by 48 patches
with mesh on refinement level 2.}
\label{fig:toy}
\end{figure}

We choose the desired effective tensor
\[\boldsymbol{B}_6 = 
\begin{bmatrix}0.82 & & \\ & 0.78 & \\ & & 0.74\end{bmatrix}
\] and employ 
$p=200$ displacement fields. We require 8 gradient descent
steps to reduce the value of the shape functional 
below $10^{-5}$. The resulting shape is illustrated in 
Figure~\ref{fig:toy_result_1}. It can be observed that, 
instead of changing the size of the cube, the size of 
the holes is changed in order to account for the 
anisotropic effective tensor.

\begin{figure}[hbt]
\begin{subfigure}{0.32\textwidth}
\includegraphics[scale=0.19]{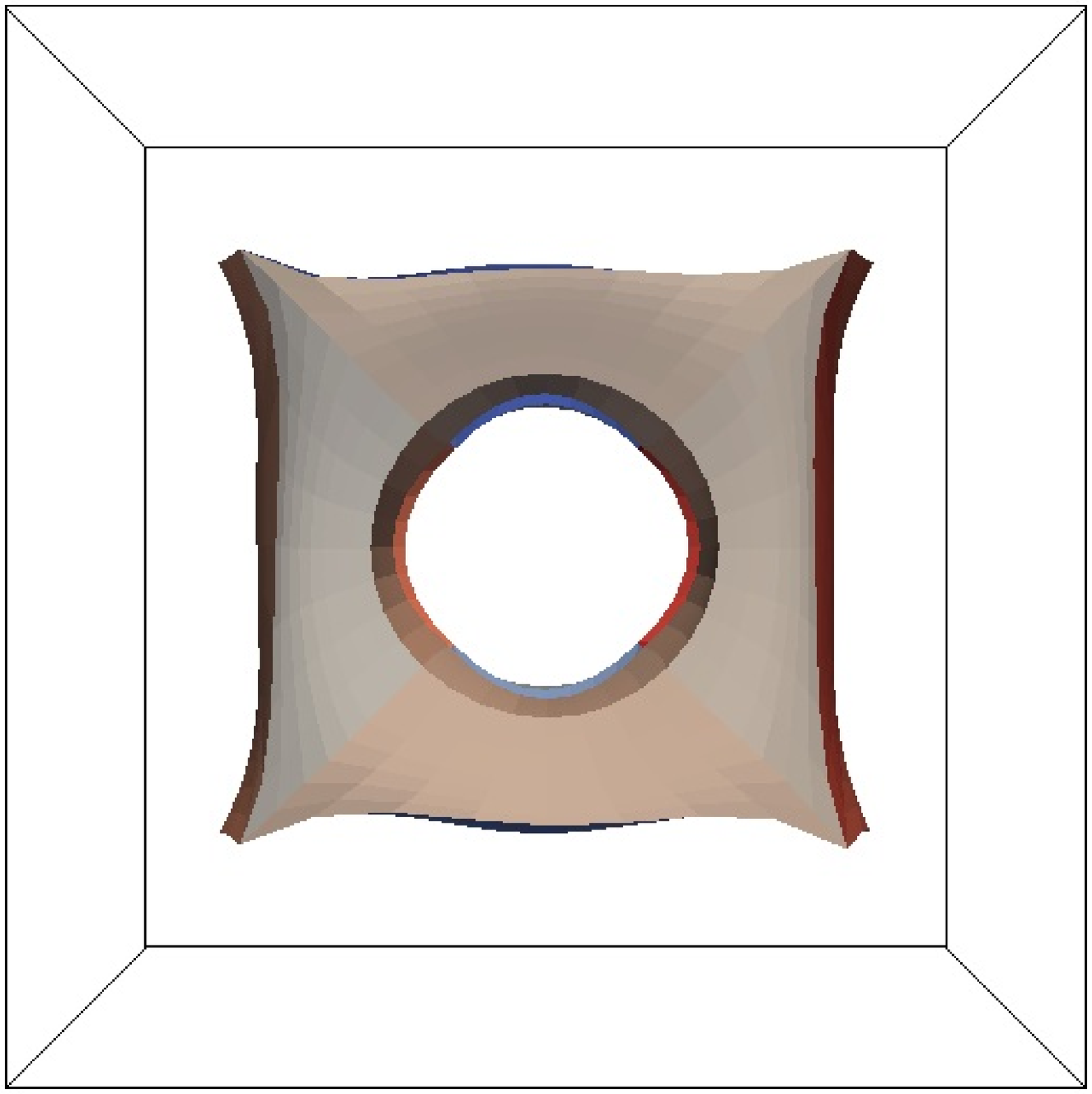}
\end{subfigure}
\begin{subfigure}{0.32\textwidth}
\includegraphics[scale=0.19]{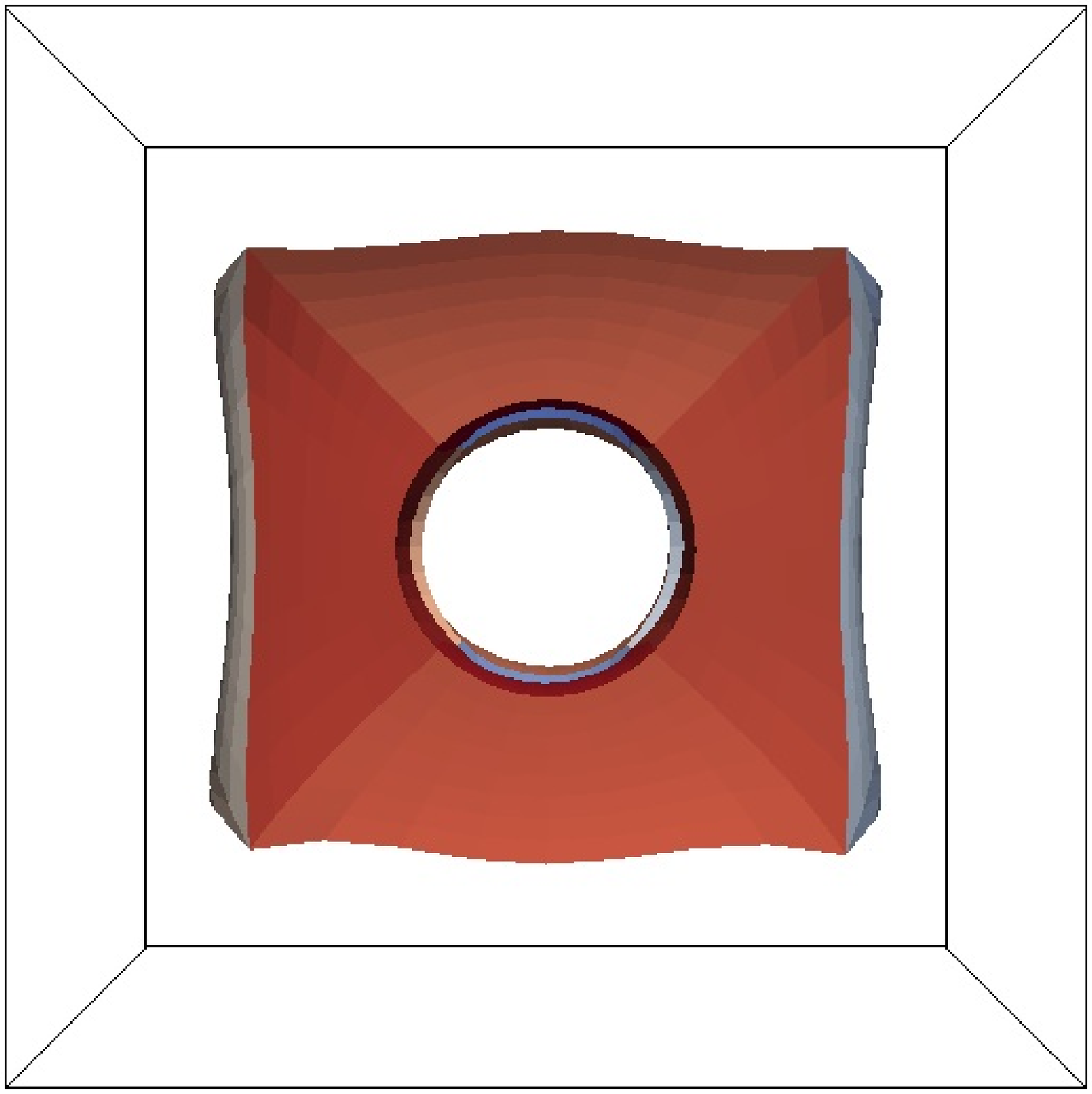}
\end{subfigure}
\begin{subfigure}{0.32\textwidth}
\includegraphics[scale=0.19]{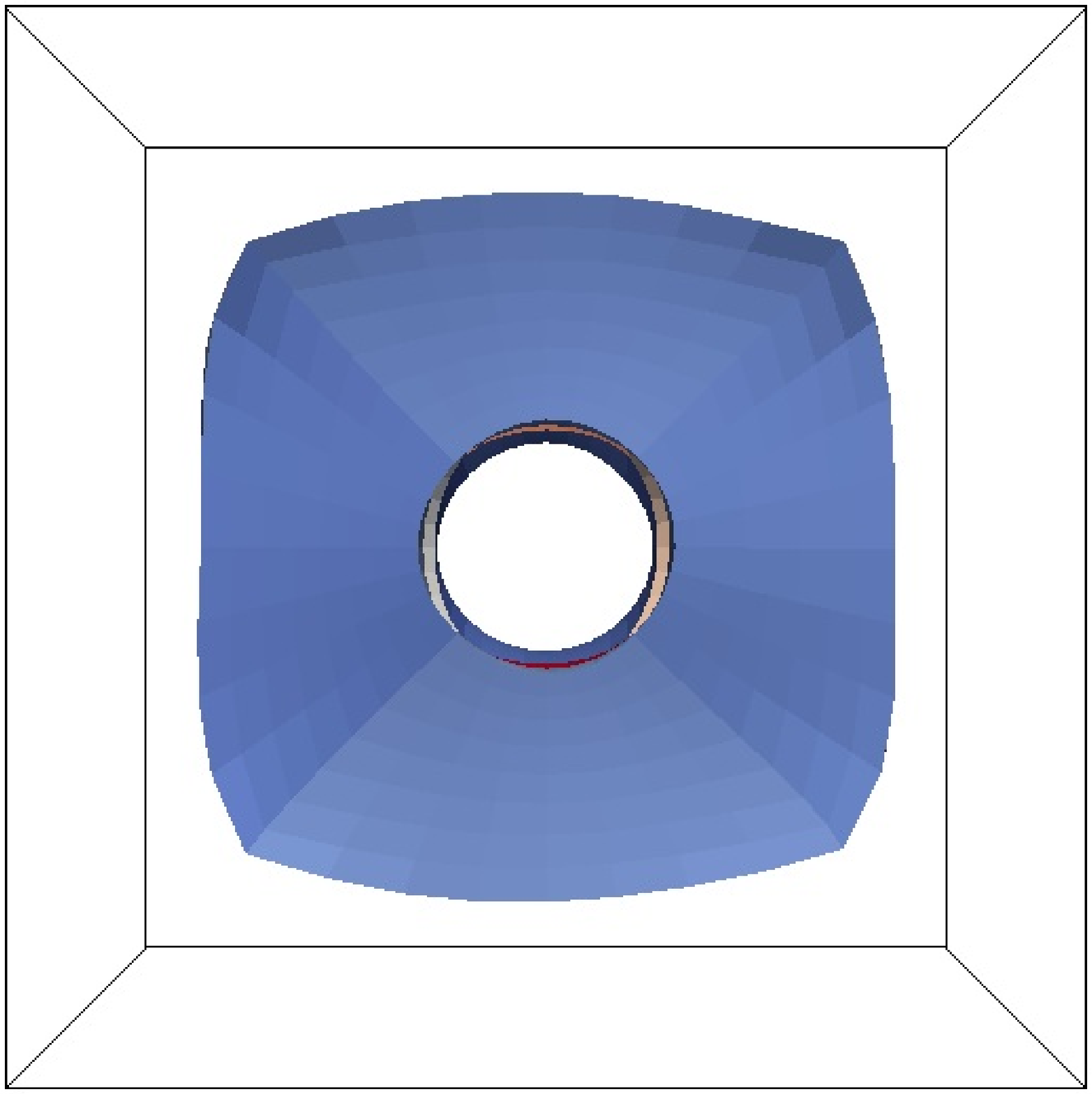}
\end{subfigure}
\caption{Drilled cube as initial guess for the desired 
tensor $\boldsymbol{B}_6$.
The number of deformation vectors is $p=200$ and 
convergence is achieved after 8 gradient steps.}
\label{fig:toy_result_1}
\end{figure}

If we reduce the number of shape deformation
vectors to $p=50$, then the result is different. The latter
situation is depicted in Figure~\ref{fig:toy_result_2}. Here,
we observe that the drilled cube (including the drill holes) 
is basically just rescaled to a cuboid which reflects the 
desired anisotropic effective tensor.

\begin{figure}[hbt]
\begin{subfigure}{0.32\textwidth}
\includegraphics[scale=0.19]{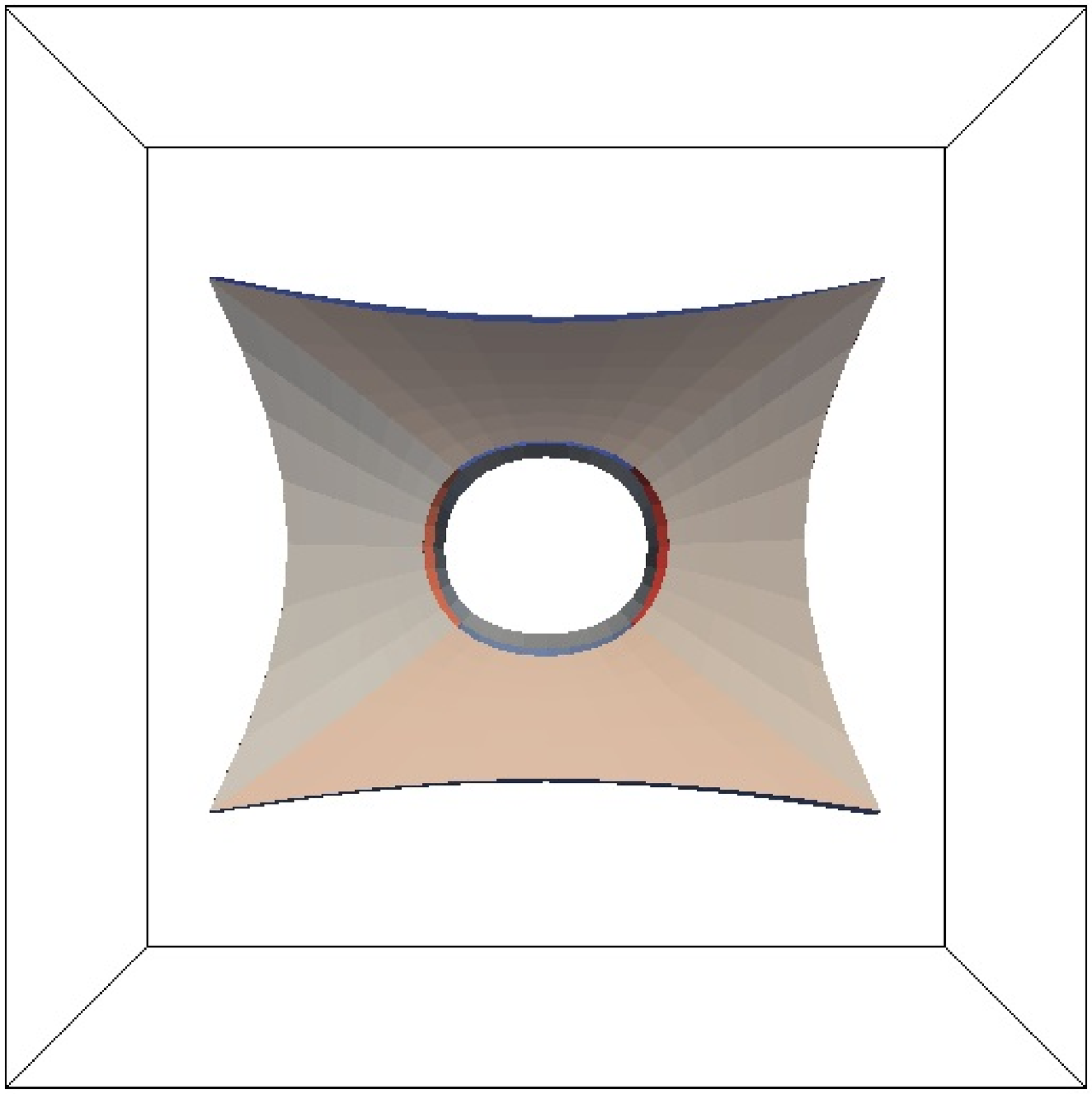}
\end{subfigure}
\begin{subfigure}{0.32\textwidth}
\includegraphics[scale=0.19]{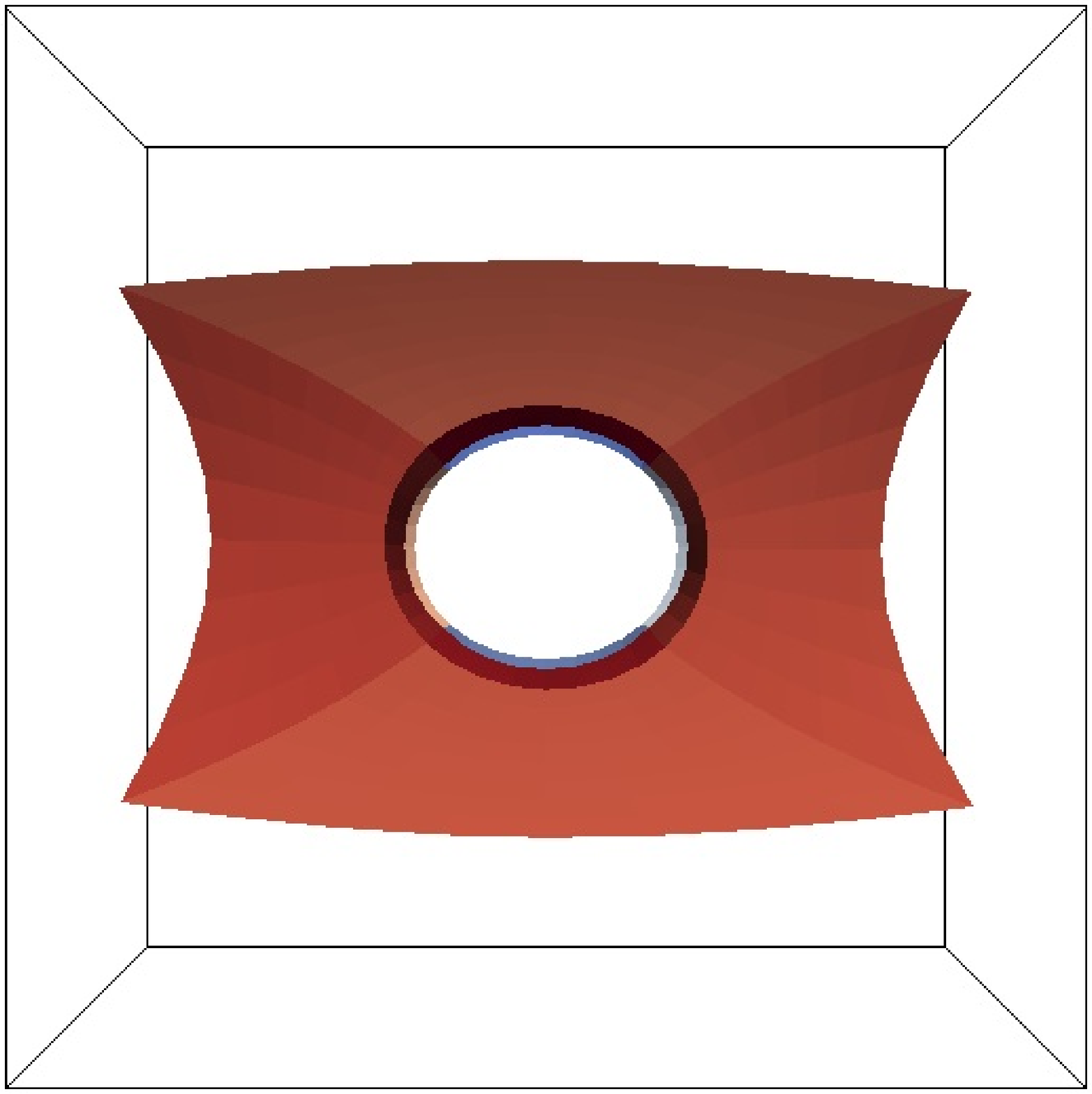}
\end{subfigure}
\begin{subfigure}{0.32\textwidth}
\includegraphics[scale=0.19]{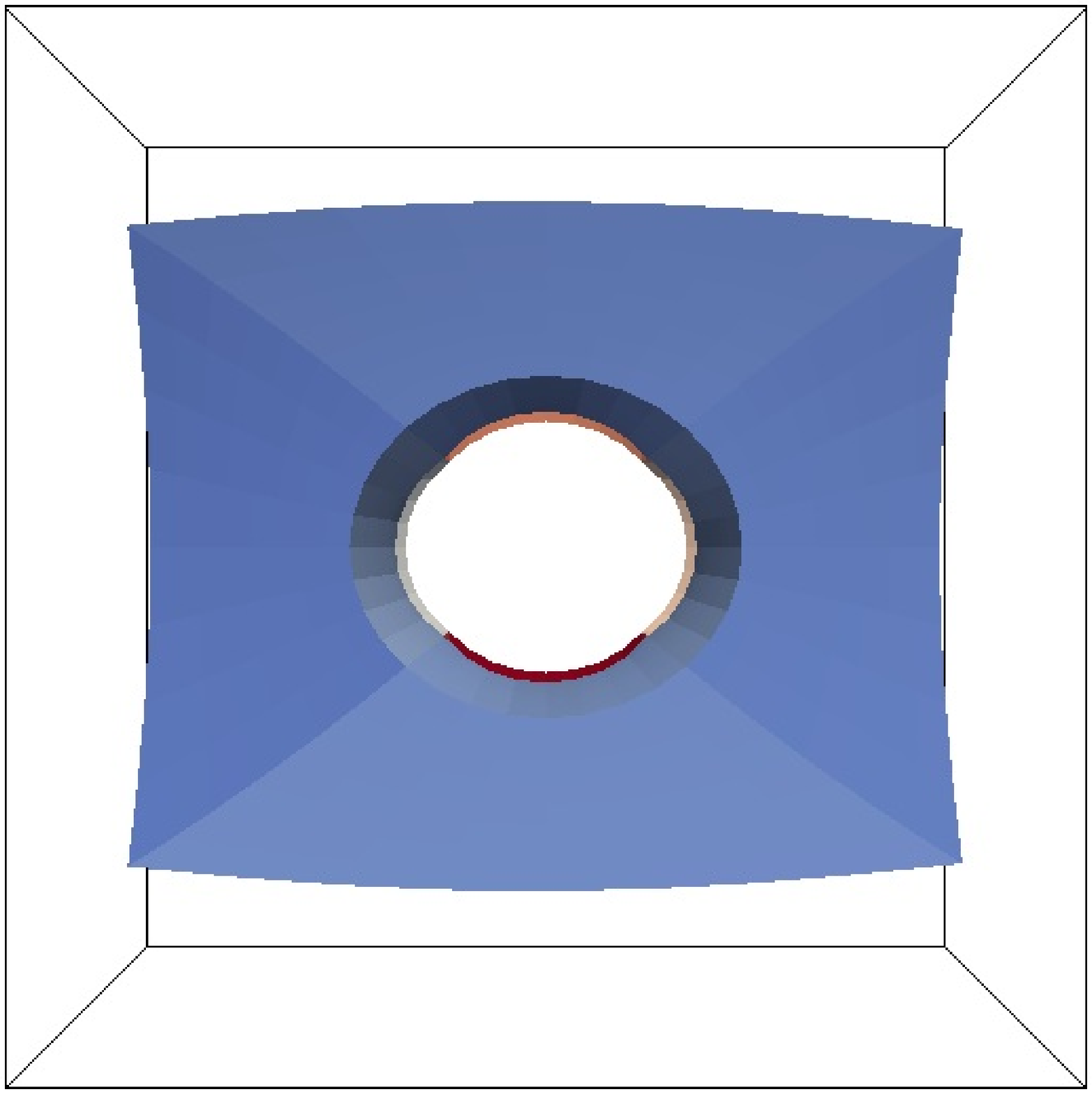}
\end{subfigure}
\caption{Drilled cube as initial guess for the desired 
tensor $\boldsymbol{B}_6$.
The number of deformation vectors is $p=50$ and 
convergence is achieved after 12 gradient steps.}
\label{fig:toy_result_2}
\end{figure}

%============================================
\section{Conclusion}\label{sct:conclusion}
%============================================
In this article, we have presented an isogeometric approach to the
shape optimization of scaffold structures. By defining the shape 
deformations with the help of a covariance kernel, we are able 
to optimize also non-smooth shapes with edges and corners
as well as disconnected shapes and shapes of genus larger 
than 0 without any additional effort. This implies a huge progress 
compared to previous approaches in shape optimization.
In combination with an isogeometric boundary element 
method, we arrive hence at a powerful new methodology 
for the numerical solution of shape optimization problems.
%============================================

\end{document}